\documentclass[12pt,twoside]{amsart}
\usepackage{geometry}
\usepackage{amssymb,amsmath,amsthm, amscd, enumerate, mathrsfs}
\usepackage{graphicx, hhline}
\usepackage[all]{xy}
\usepackage[dvipdfmx]{hyperref}

\title
{On quasi-log structures for complex analytic spaces} 
\author{Osamu Fujino}
\date{2025/2/1, version 0.22}
\address{Department of 
Mathematics, Graduate School of Science, 
Kyoto University, Kyoto 606-8502, Japan}
\email{fujino@math.kyoto-u.ac.jp}

\DeclareMathOperator{\Ass}{Ass}
\DeclareMathOperator{\Nqklt}{Nqklt}
\DeclareMathOperator{\Nklt}{Nklt}
\DeclareMathOperator{\Supp}{Supp}
\DeclareMathOperator{\NE}{\overline{NE}}
\DeclareMathOperator{\codim}{codim}
\DeclareMathOperator{\WDiv}{WDiv}
\DeclareMathOperator{\Exc}{Exc}
\DeclareMathOperator{\ncp}{ncp}
\DeclareMathOperator{\nc}{nc}
\DeclareMathOperator{\Hom}{Hom}
\DeclareMathOperator{\snc}{snc}
\DeclareMathOperator{\snct}{snc2}
\DeclareMathOperator{\Pic}{Pic}
\DeclareMathOperator{\Nqlc}{Nqlc}
\DeclareMathOperator{\Bs}{Bs}
\DeclareMathOperator{\Nlc}{Nlc}
\DeclareMathOperator{\NLC}{NLC}
\DeclareMathOperator{\Sing}{Sing}

\newtheorem{thm}{Theorem}[section]
\newtheorem{lem}[thm]{Lemma}
\newtheorem{cor}[thm]{Corollary}

\newtheorem{cla}{Claim}
\newtheorem*{claim}{Claim}

\theoremstyle{definition}
\newtheorem{step}{Step}
\newtheorem{defn}[thm]{Definition}
\newtheorem{rem}[thm]{Remark}
\newtheorem{ex}[thm]{Example}
\newtheorem*{ack}{Acknowledgments}  
\newtheorem{say}[thm]{}

\makeatletter
    
    \@addtoreset{equation}{section}
\makeatother

\begin{document}

\begin{abstract} 
We introduce the notion of quasi-log complex analytic spaces 
and establish various fundamental properties. 
Moreover, we prove that a semi-log canonical pair 
naturally has a quasi-log complex analytic space structure. 
This paper is part of the author's 
project to establish a minimal model theory for projective 
morphisms between complex analytic spaces.  
\end{abstract}

\maketitle 

\tableofcontents

\section{Introduction}\label{a-sec1}

The main purpose of this paper is to introduce the notion 
of {\em{quasi-log complex analytic spaces}} and 
establish various fundamental results. 
They will play a crucial role in the study of highly singular 
complex analytic spaces. 
Let us see the definition of quasi-log complex analytic spaces, 
which may look artificial. 

\begin{defn}[Quasi-log complex analytic spaces, see 
Definition \ref{a-def4.1}]\label{a-def1.1}
A {\em{quasi-log complex analytic space}} 
\begin{equation*}
\bigl(X, \omega, f\colon (Y, B_Y)\to X\bigr)
\end{equation*}
is a complex analytic space $X$ endowed with an 
$\mathbb R$-line bundle (or a globally $\mathbb R$-Cartier divisor) 
$\omega$ on $X$, a closed analytic subspace 
$X_{-\infty}\subsetneq X$, and a finite collection $\{C\}$ of 
reduced 
and irreducible closed analytic subspaces of $X$ such that there 
exists a 
projective morphism $f\colon (Y, B_Y)\to X$ from an 
analytic globally 
embedded simple 
normal crossing pair $(Y, B_Y)$ 
satisfying the following properties: 
\begin{itemize}
\item[(1)] $f^*\omega\sim_{\mathbb R}K_Y+B_Y$. 
\item[(2)] The natural map 
$\mathcal O_X
\to f_*\mathcal O_Y(\lceil -(B_Y^{<1})\rceil)$ 
induces an isomorphism 
\begin{equation*}
\mathcal I_{X_{-\infty}}\overset{\simeq}
{\longrightarrow} f_*\mathcal O_Y(\lceil 
-(B_Y^{<1})\rceil-\lfloor B_Y^{>1}\rfloor),  
\end{equation*}
where $\mathcal I_{X_{-\infty}}$ is the defining ideal sheaf of 
$X_{-\infty}$ on $X$. 
\item[(3)] The collection of 
closed analytic subvarieties $\{C\}$ coincides with the $f$-images 
of $(Y, B_Y)$-strata that are not included in $X_{-\infty}$. 
\end{itemize}
\end{defn}

Since we treat $\mathbb R$-line bundles and 
globally $\mathbb R$-Cartier divisors on (not necessarily 
compact) complex analytic spaces, we need the 
following remark. 

\begin{rem}[$\mathbb R$-line bundles and 
globally $\mathbb R$-Cartier divisors]\label{a-rem1.2}
Let $X$ be a complex analytic space and let 
$\Pic(X)$ be the group of line bundles on $X$, 
that is, the {\em{Picard group}} of $X$. 
An element of $\Pic(X)\otimes _{\mathbb Z}\mathbb R$ 
(resp.~$\Pic(X)\otimes _{\mathbb Z}\mathbb Q$) is called 
an {\em{$\mathbb R$-line bundle}} 
(resp.~a {\em{$\mathbb Q$-line bundle}}) on $X$. 
In this paper, we write 
the group law of $\Pic(X)\otimes _{\mathbb Z}\mathbb R$ 
additively for 
simplicity of notation. 
A {\em{globally $\mathbb R$-Cartier}} 
(resp.~{\em{globally $\mathbb Q$-Cartier}}) {\em{divisor}} is a finite 
$\mathbb R$-linear (resp.~$\mathbb Q$-linear) 
combination of Cartier divisors. 
If $\omega$ is a globally $\mathbb R$-Cartier (resp.~$\mathbb Q$-Cartier) 
divisor in Definition \ref{a-def1.1}, then we can naturally see 
$\omega$ as an $\mathbb R$-line bundle (resp.~a $\mathbb Q$-line 
bundle) on $X$. 
In Definition \ref{a-def1.1}, we always assume that $B_Y$ is a globally 
$\mathbb R$-Cartier divisor on $Y$ implicitly. 
This assumption is harmless to applications 
because $Y$ is usually a relatively compact open subset 
of a given complex analytic space. 
In that case, the support of $B_Y$ has only finitely many irreducible 
components and 
then $B_Y$ automatically becomes globally $\mathbb R$-Cartier. 
Under the assumption that $B_Y$ is globally $\mathbb R$-Cartier, 
$K_Y+B_Y$ naturally defines an 
$\mathbb R$-line bundle on $Y$. 
The condition $f^*\omega\sim _{\mathbb R} K_Y+B_Y$ in Definition 
\ref{a-def1.1} (1) means that $f^*\omega=K_Y+B_Y$ holds in 
$\Pic(Y)\otimes _{\mathbb Z}\mathbb R$. 
\end{rem}

Note that the notion of quasi-log schemes was first 
introduced by Ambro in \cite{ambro}. Definition \ref{a-def1.1} is an 
analytic counterpart of the notion of quasi-log schemes. 
For the details of the theory of 
quasi-log schemes, see \cite[Chapter 6]{fujino-foundations} 
and \cite{fujino-quasi}. A gentle introduction to the 
theory of quasi-log schemes is \cite{fujino-introduction}. 
As in the algebraic case, we establish the following theorems. 

\begin{thm}[{Adjunction, see Theorem \ref{a-thm4.4}}]\label{a-thm1.3} 
Let 
\begin{equation*}
\bigl(X, \omega, f\colon (Y, B_Y)\to X\bigr)
\end{equation*} 
be a quasi-log complex analytic space 
and let $X'$ be the union of 
$X_{-\infty}$ with a union of some 
qlc strata of $[X, \omega]$. 
Then, 
after replacing $X$ with any relatively compact open subset of $X$,  
we can construct a projective morphism $f'\colon (Y', B_{Y'})\to X'$ 
from an analytic globally embedded simple normal crossing 
pair $(Y', B_{Y'})$ such that 
\begin{equation*}
\bigl(X', \omega', f'\colon (Y', B_{Y'})\to X'\bigr)
\end{equation*} 
is a quasi-log complex analytic space with $\omega'=\omega|_{X'}$ and 
$X'_{-\infty}=X_{-\infty}$. Moreover, the qlc strata of $[X', \omega']$ are 
exactly the qlc strata of $[X, \omega]$ that are included in $X'$. 
\end{thm}

\begin{thm}[{Vanishing theorem, see Theorems \ref{a-thm4.7} and 
\ref{a-thm4.8}}]\label{a-thm1.4} 
Let 
\begin{equation*}
\bigl(X, \omega, f\colon (Y, B_Y)\to X\bigr)
\end{equation*} 
be a quasi-log complex analytic space 
and let $X'$ be the union of 
$X_{-\infty}$ with a union of some 
qlc strata of $[X, \omega]$. 
Let $\pi\colon X\to S$ be a projective morphism between 
complex analytic spaces and 
let $\mathcal L$ be a line bundle on $X$ such that 
$\mathcal L-\omega$ is nef over $S$ and $(\mathcal L-\omega)|_C$ is 
big over $\pi(C)$ for every qlc stratum $C$ of $[X, \omega]$ 
which is not contained in $X'$. 
Then 
\begin{equation*}
R^i\pi_*(\mathcal I_{X'}
\otimes \mathcal L)=0
\end{equation*} 
holds for every $i>0$, where $\mathcal I_{X'}$ is the 
defining ideal sheaf of $X'$ on $X$. 
In particular, if $\mathcal L-\omega$ is ample over $S$, 
then 
\begin{equation*}
R^i\pi_*(\mathcal I_{X'}
\otimes \mathcal L)=0
\end{equation*} 
holds for every $i>0$. 
\end{thm}

Although the definition of quasi-log complex analytic spaces looks 
complicated and artificial, 
we think that the following example shows that 
it is natural. 

\begin{ex}[Normal pairs]\label{a-ex1.5}
Let $\pi\colon X\to S$ be a projective morphism of complex 
analytic spaces such that $X$ is a normal complex variety and let $\Delta$ be 
an effective $\mathbb R$-divisor on $X$ such that 
$K_X+\Delta$ is $\mathbb R$-Cartier. 
We sometimes call $(X, \Delta)$ a {\em{normal pair}}. 
We replace $S$ with any relatively compact open subset 
of $S$. Then we can construct a projective bimeromorphic 
morphism $f\colon Y\to X$ with 
\begin{equation*}
K_Y+B_Y:=f^*(K_X+\Delta)
\end{equation*} 
such that $Y$ is smooth and $\Supp B_Y$ is a simple 
normal crossing divisor on $Y$. 
Then 
\begin{equation*}
f^*(K_X+\Delta)\sim _{\mathbb R} K_Y+B_Y
\end{equation*} 
obviously holds 
and 
\begin{equation*}
\mathcal J_{\NLC}(X, \Delta):=f_*\mathcal O_Y(\lceil -(B^{<1}_Y)\rceil 
-\lfloor B^{>1}_Y\rfloor)
\end{equation*} 
is a well-defined coherent ideal sheaf on $X$ which defines the non-lc locus 
$\Nlc(X, \Delta)$ of $(X, \Delta)$. By definition, 
$C$ is a log canonical center of $(X, \Delta)$ if and only if 
$C$ is not contained in $\Nlc(X, \Delta)$ and is the $f$-image 
of some log canonical center of $(Y, B_Y)$. 
Hence, 
\begin{equation*}
\left(X, K_X+\Delta, f\colon (Y, B_Y)\to X\right)
\end{equation*} 
with $X_{-\infty}:=\Nlc(X, \Delta)$ satisfies the conditions 
in Definition \ref{a-def1.1}, 
that is, 
\begin{equation*}
\left(X, K_X+\Delta, f\colon (Y, B_Y)\to X\right)
\end{equation*}  
is a quasi-log complex analytic space. 
By construction, $X_{-\infty}=\emptyset$ if and only if 
$(X, \Delta)$ is log canonical. 
\end{ex}

In this paper, after we define quasi-log complex analytic spaces 
and prove the adjunction formula and vanishing theorems 
for them (see Section \ref{a-sec4}), we establish the basepoint-free 
theorem (see Theorem \ref{a-thm6.1}), the basepoint-freeness of 
Reid--Fukuda type (see Theorem \ref{a-thm7.1}), 
the effective freeness (see Theorems \ref{a-thm8.1} and 
\ref{a-thm8.2}), the cone and 
contraction theorem (see Theorem \ref{a-thm9.2}), 
and so on, for quasi-log complex 
analytic spaces. We can use them for the 
study of normal pairs by Example \ref{a-ex1.5}. 
We note that the cone and contraction theorem 
for normal pairs, which 
is sufficient for the minimal model 
program for log canonical pairs, 
was already proved in the complex 
analytic setting in \cite{fujino-cone-contraction}. 
We do not need the framework of quasi-log complex analytic 
spaces in \cite{fujino-cone-contraction}. 
However, it seems to be difficult to prove 
the basepoint-free theorem of Reid--Fukuda type 
for normal pairs in the 
complex analytic setting without using the framework of 
quasi-log complex analytic spaces. 
By combining some results obtained in this paper 
with Example \ref{a-ex1.5}, we have: 

\begin{thm}[Effective freeness of Reid--Fukuda type for 
log canonical pairs]\label{a-thm1.6}
Let $\pi\colon X\to S$ be a projective morphism 
of complex analytic spaces such that 
$(X, \Delta)$ is log canonical and that $\Delta$ is a $\mathbb Q$-divisor. 
Let $\mathcal L$ be a $\pi$-nef line bundle 
on $X$ such that $a\mathcal L-(K_X+\Delta)$ 
is nef and log big over $S$ with respect to 
$(X, \Delta)$ for some positive real number $a$. 
This means that $a\mathcal L-(K_X+\Delta)$ is nef and 
big over $S$ and that $(a\mathcal L-(K_X+\Delta))|_C$ is 
big over $\pi(C)$ for every log canonical center $C$ of $(X, \Delta)$. 
Then there exists a positive integer $m_0$, which 
depends only on $\dim X$ and $a$, 
such that $\mathcal L^{\otimes m}$ 
is $\pi$-generated for every $m\geq m_0$. 
Moreover, we may allow $\Delta$ to be an $\mathbb R$-divisor 
when 
$a\mathcal L-(K_X+\Delta)$ is $\pi$-ample over $S$ 
in the above statement. 
\end{thm}

Theorem \ref{a-thm1.6} is a generalization of 
\cite[Theorem 2.2.4]{fujino-effective-tohoku}. 
We note that we do not have to replace $S$ with 
a relatively compact open subset of $S$ in Theorem \ref{a-thm1.6}. 
The notion of quasi-log complex analytic spaces 
is very useful for the proof of Theorem \ref{a-thm1.6}. 
The author does not know how to prove Theorem \ref{a-thm1.6} 
in the framework of \cite{fujino-cone-contraction}. 
Precisely speaking, we first establish the basepoint-free 
theorem for quasi-log complex analytic spaces (see 
Theorem \ref{a-thm6.1}). 
Then, by using it, we prove the basepoint-free theorem 
of Reid--Fukuda type for quasi-log 
complex analytic spaces (see Theorem \ref{a-thm7.1}). 
Here, the framework of quasi-log complex analytic spaces 
plays an important role. 
Finally, we obtain the effective freeness 
for complex analytic quasi-log canonical pairs in 
Theorems \ref{a-thm8.1} and \ref{a-thm8.2}. 
By combining it with Example \ref{a-ex1.5}, 
we have Theorem \ref{a-thm1.6}. 
Moreover, we think that 
we need the theory of quasi-log complex 
analytic spaces for the study of semi-log canonical pairs 
in the complex analytic setting by the following theorem. 

\begin{thm}[Semi-log canonical pairs, 
see Theorem \ref{a-thm10.1}]\label{a-thm1.7}
Let $\pi\colon X\to S$ be a projective morphism 
of complex analytic spaces and 
let $(X, \Delta)$ be a semi-log canonical pair. 
Then, after replacing $S$ with any relatively compact 
open subset of $S$, 
$[X, K_X+\Delta]$ naturally becomes a quasi-log complex 
analytic space such that 
$\Nqlc(X, K_X+\Delta)=\emptyset$ and 
that $C$ is a qlc center of $[X, K_X+\Delta]$ if and 
only if $C$ is a semi-log canonical center of $(X, \Delta)$. 

More precisely, after replacing $S$ with any relatively compact 
open subset of $S$, we can construct a projective surjective morphism 
$f\colon (Y, B_Y)\to X$ from an analytic globally embedded 
simple normal crossing pair $(Y, B_Y)$ such that 
the natural map 
\begin{equation*}
\mathcal O_X\to f_*\mathcal O_Y(\lceil 
-(B^{<1}_Y)\rceil)
\end{equation*} 
is an isomorphism 
and that $C$ is 
the $f$-image of some stratum of $(Y, B_Y)$ 
if and only if 
$C$ is a semi-log canonical 
center of $(X, \Delta)$ 
or an irreducible component of $X$. 
Moreover, if every irreducible component of $X$ has 
no self-intersection in codimension one, then we can 
make $f\colon Y\to X$ bimeromorphic. 
\end{thm}

Theorem \ref{a-thm1.7} is obviously a complex analytic generalization 
of \cite[Theorem 1.2]{fujino-fundamental-slc}. 
Example \ref{a-ex1.8} may help us understand 
Theorem \ref{a-thm1.7}. 

\begin{ex}\label{a-ex1.8}
We consider 
\begin{equation*}
X:=\left(X_0X_2^2-X_1^2(X_1-1)=0\right)\subset \mathbb P^2. 
\end{equation*} 
Then $(X, 0)$ is a projective semi-log canonical curve. 
Let $\alpha\colon M\to \mathbb P^2$ be the 
blow-up at $[1:0:0]\in \mathbb P^2$. 
We put $Y:=X'+E$, 
where $X'$ is the strict transform of $X$ on $M$ and 
$E$ is the $\alpha$-exceptional curve. 
Then it is easy to see that $Y$ is a simple normal crossing 
divisor on $M$, 
\begin{equation}\label{a-eq1.1}
\alpha^*(K_{\mathbb P^2}+X)=K_M+X'+E, 
\end{equation} 
and $f_*\mathcal O_Y\simeq \mathcal O_X$, where 
$f:=\alpha|_Y$. 
By \eqref{a-eq1.1} and adjunction, 
\begin{equation*}
f^*K_X=K_Y 
\end{equation*} 
holds. Thus 
\begin{equation*}
\left(X, K_X, f\colon (Y, 0)\to X\right)
\end{equation*} 
is a quasi-log complex analytic space with $X_{-\infty}=\emptyset$. 
We note that $X$ is irreducible but $Y$ is reducible. 
In particular, $f\colon Y\to X$ is not bimeromorphic. 
\end{ex}

By Theorem \ref{a-thm1.7}, we can use the results established 
for quasi-log complex analytic spaces in this paper to study 
semi-log canonical pairs. 
Of course, by combining Theorems \ref{a-thm8.1} and 
\ref{a-thm8.2} 
with Theorem \ref{a-thm1.7}, we see that Theorem \ref{a-thm1.6} 
holds for complex analytic semi-log canonical pairs. 
More precisely, the basepoint-free theorem and 
its variants hold true for semi-log canonical 
pairs in the complex analytic setting. 
Although we do not state it explicitly here, 
the cone and contraction theorem holds in full generality 
for complex analytic semi-log canonical pairs. 
We note that this paper is not self-contained. 
We strongly recommend that 
the reader looks at \cite{fujino-cone-contraction} before reading 
this paper. Roughly speaking, this paper explains how to 
use the strict support condition and the vanishing theorems 
established in \cite{fujino-analytic-vanishing} systematically 
by introducing the framework of quasi-log complex analytic spaces.  

We briefly summarize the current state of 
the minimal model theory for projective morphisms between complex analytic spaces.
 
\begin{rem}[Minimal model program for projective 
morphisms of complex analytic spaces]\label{a-rem1.9} 
We are mainly interested in projective morphisms 
between complex analytic spaces. 
Roughly speaking, in \cite{fujino-minimal}, we 
translated \cite{bchm} and \cite{hacon-mckernan} into 
the complex analytic setting. 
After \cite{fujino-minimal}, Das, Hacon, and P\u aun gave an 
alternative approach to the minimal model program 
of kawamata log terminal pairs for projective 
morphisms between complex analytic spaces (see \cite{dhp}). 
Moreover, Lyu and Murayama established a new approach to the relative 
minimal model program in \cite{lyu-murayama}, 
which can work in larger categories of spaces. 
We note that the 
framework of the minimal model program established in 
\cite{nakayama1} and \cite{nakayama2} is almost sufficient for 
\cite{fujino-minimal}. We do not need \cite{fujino-analytic-vanishing}, 
which is an analytic generalization of \cite[Chapter 5]{fujino-foundations}, for 
\cite{fujino-minimal}. The ACC for log canonical 
thresholds in the complex analytic setting (see 
\cite{fujino-acc}) is 
an easy consequence of \cite{fujino-minimal} and 
\cite{hacon-mckernan-xu}. We can use \cite{fujino-minimal} to 
prove the inversion of adjunction of log canonicity for 
complex analytic spaces (see \cite{fujino-inversion}). 
On the other hand, 
\cite{fujino-cone-contraction} 
and this paper heavily depend on 
\cite{fujino-analytic-vanishing}. We note 
that \cite{fujino-cone-contraction} is a complex analytic generalization 
of \cite{fujino-fundamental} based on \cite{fujino-analytic-vanishing}. 
This paper explains how to generalize \cite[Chapter 6]{fujino-foundations},  
\cite{fujino-fundamental-slc}, 
\cite{fujino-reid-fukuda}, and 
\cite{fujino-effective} into the complex analytic setting. 
In \cite{enokizono-hashizume1} and \cite{enokizono-hashizume2}, 
based on \cite{fujino-minimal}, 
Enokizono and Hashizume discussed the minimal model 
program for log canonical pairs in the complex analytic setting. 
In \cite{fujino-abundance}, which is an analytic generalization of \cite{fujino-master} and 
\cite{fujino-gongyo}, we studied the abundance conjecture for projective morphisms 
of complex analytic spaces. 
Finally, Enokizono and Hashizume strengthened some results of \cite{enokizono-hashizume2} 
in \cite{enokizono-hashizume3} and \cite{hashizume}. 
By the above mentioned works, 
we see that almost all conjectures of the minimal 
model theory for projective morphisms between 
complex analytic spaces follow from the original conjectures 
for projective varieties. 
\end{rem}

We make a remark on \cite{fujino-fundamental-slc} for the 
reader's convenience. 

\begin{rem}\label{a-rem1.10}
Note that \cite[Definition A.20]{fujino-fundamental-slc} 
has some subtle troubles. For the details, 
see \cite[Definition 2.1.25, Remark 2.1.16, and 
Lemma 2.1.18]{fujino-foundations}. 
The proof of \cite[Theorem 1.12]{fujino-fundamental-slc} 
is insufficient. For the details, 
see Theorem \ref{a-thm10.4} and Remark \ref{a-rem10.5} below. 
\end{rem}

We look at the organization of this paper. 
In Section \ref{a-sec2}, we collect some basic definitions and 
results necessary for this paper. 
In Section \ref{a-sec3}, we recall the 
strict support condition and the vanishing theorems 
for analytic simple normal crossing pairs established in 
\cite{fujino-analytic-vanishing}. 
Note that we do not prove them in this paper. 
In Section \ref{a-sec4}, which is the main part of 
this paper, we introduce the notion of 
quasi-log complex analytic spaces and 
prove some basic properties. 
In Section \ref{a-sec5}, we prepare several useful lemmas. 
Although they may look complicated and artificial, 
they are very 
important.  
In Section \ref{a-sec6}, we prove the basepoint-free theorem 
for quasi-log complex analytic spaces. 
Then, in Section \ref{a-sec7}, we prove the basepoint-free 
theorem of Reid--Fukuda type for quasi-log complex 
analytic spaces. 
In Section \ref{a-sec8}, we establish the effective 
basepoint-freeness and effective very ampleness for 
quasi-log complex analytic spaces. The argument in this section 
is new and is slightly simpler than the known one. 
In Section \ref{a-sec9}, we discuss the cone and contraction 
theorem for quasi-log complex analytic spaces. 
In Subsection \ref{a-subsec9.1}, we prove that any 
extremal ray is spanned by a rational curve. 
In Section \ref{a-sec10}, we treat complex analytic semi-log 
canonical pairs. 
Roughly speaking, we show that a semi-log canonical pair 
naturally becomes a quasi-log complex analytic space. 
In Subsection \ref{a-subsec10.1}, 
we explain some vanishing theorems for the reader's convenience. 
In Subsection \ref{a-subsec10.2}, 
we briefly discuss Shokurov's polytopes 
for semi-log canonical pairs for the sake of completeness. 

\begin{ack}\label{a-ack}
The author was partially 
supported by JSPS KAKENHI Grant Numbers 
JP19H01787, JP20H00111, JP21H00974, JP21H04994, JP23K20787. 
He thanks Professor Taro Fujisawa for always 
giving him warm encouragement. 
He also thanks Yoshinori Gongyo very much. 
\end{ack}

In this paper, we assume that 
every complex analytic space is 
{\em{Hausdorff}} and {\em{second-countable}}. 
An irreducible and reduced complex analytic space 
is called a {\em{complex variety}}. 
We will freely use the basic definitions and 
results on complex analytic geometry 
in \cite{banica} and \cite{fischer}. 
Nakayama's book \cite{nakayama2} may be helpful. 
We will also freely use Serre's GAGA (see \cite{serre}) 
throughout this paper. 
We strongly recommend that the reader looks at 
\cite{fujino-cone-contraction} before reading this paper. 
This paper is a continuation of \cite{fujino-cone-contraction} and 
is also a supplement to \cite{fujino-cone-contraction}. 

\section{Preliminaries}\label{a-sec2}
In this section, we will recall some basic definitions and 
properties of complex analytic spaces necessary for 
subsequent sections. 
For the details, see \cite{fujino-minimal} and 
\cite[Sections 2.1, 4.4, and 4.5]{fujino-cone-contraction}. 

\begin{say}[Hybrids of $\mathbb R$-line bundles and 
globally $\mathbb R$-Cartier divisors]\label{a-say2.1}
As we already mentioned in Remark \ref{a-rem1.2}, 
we usually treat hybrids of $\mathbb R$-line bundles and 
globally $\mathbb R$-Cartier divisors on a complex 
analytic space $X$. 
Note that a {\em{globally $\mathbb R$-Cartier divisor}} is a 
finite $\mathbb R$-linear combination of 
Cartier divisors. 
We often write 
\begin{equation*}
\mathcal L+A\sim _{\mathbb R} \mathcal M+B, 
\end{equation*} 
where $\mathcal L, \mathcal M\in \Pic(X)\otimes _{\mathbb Z} 
\mathbb R$, and $A$ and $B$ are globally $\mathbb R$-Cartier 
divisors on $X$. 
This means that 
\begin{equation*}
\mathcal L+\mathcal A=\mathcal M+\mathcal B
\end{equation*}
holds in $\Pic(X)\otimes _{\mathbb Z}\mathbb R$, 
where $\mathcal A$ and $\mathcal B$ are $\mathbb R$-line bundles 
naturally associated to $A$ and $B$, respectively. 
We note that we usually write the group law of $\Pic(X)\otimes 
_{\mathbb Z}\mathbb R$ additively for simplicity of 
notation. 
\end{say}

\begin{say}[Divisors]\label{a-say2.2}
Let $X$ be a reduced equidimensional complex analytic space. 
A {\em{prime divisor}} on $X$ is an irreducible 
and reduced closed analytic subspace of codimension one. 
An {\em{$\mathbb R$-divisor}} $D$ on $X$ is a formal 
sum 
\begin{equation*}
D=\sum _i a_i D_i, 
\end{equation*} 
where $D_i$ is a prime divisor on $X$ with 
$D_i\ne D_j$ for $i\ne j$, 
$a_i\in \mathbb R$ for every $i$, and the {\em{support}} 
\begin{equation*}
\Supp D:=\bigcup _{a_i\ne 0}D_i
\end{equation*} 
is a closed analytic subset of $X$. 
In other words, the formal sum $\sum _i a_i 
D_i$ is locally finite. 
If $a_i\in \mathbb Z$ (resp.~$a_i\in 
\mathbb Q$) for 
every $i$, then $D$ is called 
a {\em{divisor}} (resp.~{\em{$\mathbb Q$-divisor}}) on $X$. 
Note that a divisor is sometimes called 
an {\em{integral Weil divisor}} in 
order to emphasize the condition that $a_i\in \mathbb Z$ for every $i$. 
If $0\leq a_i\leq 1$ (resp.~$a_i\leq 1$) 
holds for every $i$, then 
an $\mathbb R$-divisor $D$ is called a {\em{boundary}} 
(resp.~{\em{subboundary}}) $\mathbb R$-divisor. 

Let $D=\sum _i a_i D_i$ be an $\mathbb R$-divisor 
on $X$ such that $D_i$ is a prime divisor 
for every $i$ with $D_i\ne D_j$ for $i\ne j$. 
The {\em{round-down}} $\lfloor D\rfloor$ of $D$ is 
defined to be the divisor 
\begin{equation*}
\lfloor D\rfloor =\sum _i \lfloor a_i\rfloor D_i, 
\end{equation*} 
where $\lfloor x\rfloor$ is 
the integer defined by $x-1<\lfloor x\rfloor \leq x$ 
for every real number $x$. 
The {\em{round-up}} and the 
{\em{fractional part}} of $D$ are defined to be 
\begin{equation*}
\lceil D \rceil :=-\lfloor -D\rfloor, \quad 
\text{and} \quad \{D\}:=D-\lfloor D\rfloor, 
\end{equation*} 
respectively. We put 
\begin{equation*}
D^{=1}:=\sum _{a_i=1}D_i, \quad 
D^{<1}:=\sum _{a_i<1} a_i D_i, \quad \text{and} \quad 
D^{>1}:=\sum _{a_i>1}a_i D_i. 
\end{equation*}

Let $D$ be an $\mathbb R$-divisor on $X$ 
and let $x$ be a point of $X$. 
If $D$ is written as a finite $\mathbb R$-linear 
(resp.~$\mathbb Q$-linear) combination of Cartier 
divisors on some open 
neighborhood of $x$, 
then $D$ is said to be {\em{$\mathbb R$-Cartier at $x$}} 
(resp.~{\em{$\mathbb Q$-Cartier at $x$}}). 
If $D$ is $\mathbb R$-Cartier 
(resp.~$\mathbb Q$-Cartier) at $x$ for every $x\in X$, 
then $D$ is said to be {\em{$\mathbb R$-Cartier}} 
(resp.~{\em{$\mathbb Q$-Cartier}}). 
Note that a $\mathbb Q$-Cartier $\mathbb R$-divisor 
$D$ is automatically a $\mathbb Q$-Cartier 
$\mathbb Q$-divisor by definition. 
If $D$ is a finite $\mathbb R$-linear (resp.~$\mathbb Q$-linear) 
combination of Cartier divisors on $X$, 
then we say that $D$ 
is a {\em{globally $\mathbb R$-Cartier $\mathbb R$-divisor}} 
(resp.~{\em{globally $\mathbb Q$-Cartier $\mathbb Q$-divisor}}).  

Two $\mathbb R$-divisors $D_1$ and $D_2$ are said to 
be {\em{linearly equivalent}} if 
$D_1-D_2$ is a principal Cartier divisor. 
The linear equivalence is denoted by $D_1\sim D_2$. 
Two $\mathbb R$-divisors $D_1$ and $D_2$ are 
said to be {\em{$\mathbb R$-linearly equivalent}} 
(resp.~{\em{$\mathbb Q$-linearly equivalent}}) 
if $D_1-D_2$ is a {\em{finite}} $\mathbb R$-linear 
(resp.~$\mathbb Q$-linear) combination 
of principal Cartier divisors. 
When $D_1$ is $\mathbb R$-linearly (resp.~$\mathbb Q$-linearly) 
equivalent to 
$D_2$, we write $D_1\sim _{\mathbb R}D_2$ 
(resp.~$D_1\sim _{\mathbb Q}D_2$). 

\begin{rem}\label{a-rem2.3}
Let $D$ be an $\mathbb R$-Cartier $\mathbb R$-divisor 
on $X$ and let $U$ be any relatively compact open subset of $X$. 
Then it is easy to see that $D|_U$ is a globally 
$\mathbb R$-Cartier $\mathbb R$-divisor on $U$. 
\end{rem}
\end{say}

\begin{say}[Singularities of pairs]\label{a-say2.4}
We have already discussed the notion of singularities of 
pairs for complex analytic spaces in detail in 
\cite[Section 2.1]{fujino-cone-contraction}. 
Hence we omit the details here. 
We do not repeat the definitions of 
{\em{log canonical pairs}}, {\em{kawamata log terminal pairs}}, 
{\em{log canonical centers}}, and so on. 
Here we define {\em{semi-log canonical pairs}} 
for complex analytic spaces. 

Let $X$ be an equidimensional reduced complex analytic space 
that satisfies Serre's $S_2$ condition and is normal crossing 
in codimension one. Let $X^{\nc}$ be the largest open subset 
of $X$ consisting of smooth points and normal crossing 
points. Then we have an invertible dualizing sheaf $\omega_{X^{\nc}}$ 
on $X^{\nc}$. We put 
$\omega_X:=\iota_*\omega_{X^{\nc}}$, where $\iota\colon 
X^{\nc}\hookrightarrow X$, 
and call it the {\em{canonical sheaf}} of $X$. 
Since $\codim _X(X\setminus X^{\nc})\geq 2$ and $X$ 
satisfies Serre's $S_2$ condition, $\omega_X$ is a reflexive 
sheaf of rank one on $X$. 
Although we can not always define $K_X$ globally with 
$\mathcal O_X(K_X)\simeq 
\omega_X$, we use the symbol $K_X$ as a formal divisor class 
with an isomorphism $\mathcal O_X(K_X)\simeq \omega_X$ if 
there is no danger of confusion. 

\begin{defn}[Semi-log canonical pairs]\label{a-def2.5}
Let $X$ be an equidimensional reduced complex analytic 
space that satisfies Serre's $S_2$ condition 
and is normal crossing in codimension one. 
Let $\Delta$ be an effective $\mathbb R$-divisor 
on $X$ such that no irreducible component of $\Supp 
\Delta$ is contained in the singular locus of $X$. 
In this situation, 
the pair $(X, \Delta)$ is called a {\em{semi-log canonical 
pair}} (an {\em{slc pair}}, for short) if 
\begin{itemize}
\item[(1)] $K_X+\Delta$ is $\mathbb R$-Cartier, and 
\item[(2)] $(X^\nu, \Theta)$ is log canonical, where 
$\nu\colon X^\nu\to X$ is the normalization and 
$K_{X^\nu}+\Theta:=\nu^*(K_X+\Delta)$. 
\end{itemize}

Let $(X, \Delta)$ be a semi-log canonical pair. 
A closed analytic subvariety $C$ is called a {\em{semi-log 
canonical center}} (an {\em{slc center}}, for short) 
{\em{of $(X, \Delta)$}} if 
$C$ is the $\nu$-image of some log canonical center 
of $(X^\nu, \Theta)$. 
A closed subvariety $S$ is sometimes 
called an {\em{slc stratum}} 
if $S$ is an slc center of $(X, \Delta)$ or 
$S$ is an irreducible component of $X$. 
\end{defn}
Let $X$ be an equidimensional complex analytic space. A real vector 
space spanned by the prime divisors on $X$ is 
denoted by $\WDiv_{\mathbb R}(X)$, which has a canonical basis 
given by the prime divisors. Let $D$ be an element of 
$\WDiv_{\mathbb R}(X)$. Then the sup norm of $D$ with 
respect to this basis is denoted by $|\!| D|\!|$. 
Let $V$ be a finite-dimensional affine subspace of $\WDiv_{\mathbb R}(X)$, 
which is defined 
over the rationals. Let $L$ be a compact subset of $X$. 
We put 
\begin{equation*}
\mathcal L(V; L):=\{\Delta\in V\,| \, \text{$(X, \Delta)$ 
is semi-log canonical at $L$}\}. 
\end{equation*} 
Then we can check that $\mathcal L(V; L)$ is a rational polytope. 
For the details, see \cite[2.1.10]{fujino-cone-contraction}, 
where we treat the case where $X$ is normal. 
We will use $\mathcal L(V; L)$ in Subsection \ref{a-subsec10.2}. 
\end{say}

\begin{say}[Kleiman--Mori cones]\label{a-say2.6} 
Here we briefly discuss the basics about Kleiman--Mori cones 
in the complex analytic setting. For the details, see 
\cite[Section 4]{fujino-minimal} and 
\cite[Sections 4.4 and 4.5]{fujino-cone-contraction}. 

Let $\pi\colon X\to S$ be a projective 
morphism of complex analytic spaces and let 
$W$ be a compact subset of $S$. 
Let $Z_1(X/S; W)$ be the free abelian group 
generated by the projective integral curves $C$ on $X$ such that 
$\pi(C)$ is a point of $W$. 
Let $U$ be any open neighborhood of $W$. 
Then we can consider the following intersection pairing 
\begin{equation*} 
\cdot :
\Pic\!\left(\pi^{-1}(U)\right)\times Z_1(X/S; W)\to \mathbb Z
\end{equation*}  
given by $\mathcal L\cdot C\in \mathbb Z$ for 
$\mathcal L\in \Pic\!\left(\pi^{-1}(U)\right)$ and 
$C\in Z_1(X/S; W)$. 
We say that $\mathcal L$ is {\em{$\pi$-numerically 
trivial over $W$}} when $\mathcal L\cdot C=0$ for 
every $C\in Z_1(X/S; W)$. 
We take $\mathcal L_1, \mathcal L_2\in 
\Pic\!\left(\pi^{-1}(U)\right)$. 
If $\mathcal L_1\otimes \mathcal L_2^{-1}$ 
is $\pi$-numerically trivial over $W$, 
then we write $\mathcal L_1\equiv_W\mathcal L_2$ 
and say that $\mathcal L_1$ is {\em{numerically equivalent}} to 
$\mathcal L_2$ over $W$. 
We put 
\begin{equation*}
\widetilde A(U, W):=\Pic\!\left(\pi^{-1}(U)\right)/{\equiv_W}
\end{equation*}  
and define 
\begin{equation*}
A^1(X/S; W):=\underset{W\subset U}\varinjlim
\widetilde A(U, W), 
\end{equation*}  
where $U$ runs through all the open neighborhoods of 
$W$. 

From now on, we further assume that 
$A^1(X/S; W)$ is a finitely generated 
abelian group. Then we can define the {\em{relative Picard number}} 
$\rho(X/S; W)$ to be the rank of 
$A^1(X/S; W)$. 
We put 
\begin{equation*}
N^1(X/S; W):=A^1(X/S; W)\otimes _{\mathbb Z} \mathbb R. 
\end{equation*} 
Let $A_1(X/S; W)$ be the image of 
\begin{equation*} 
Z_1(X/S; W)\to \Hom_{\mathbb Z} \left(A^1(X/S; W), 
\mathbb Z\right)
\end{equation*} 
given by the above intersection pairing. 
Then we set 
\begin{equation*} 
N_1(X/S; W):=A_1(X/S; W)\otimes _{\mathbb Z}\mathbb R. 
\end{equation*} 
In this setting, we can define the {\em{Kleiman--Mori cone}} 
\begin{equation*} 
\NE(X/S; W)
\end{equation*} 
of $\pi\colon X\to S$ over $W$, that is, 
$\NE(X/S; W)$ is the closure of the convex cone in 
$N_1(X/S; W)$ spanned by the projective 
integral curves $C$ on $X$ such that 
$\pi(C)$ is a point of $W$. 
An element $\zeta\in N^1(X/S; W)$ is called 
{\em{$\pi$-nef over $W$}} or {\em{nef over $W$}} 
if $\zeta\geq 0$ on $\NE(X/S; W)$, equivalently, 
$\zeta|_{\pi^{-1}(w)}$ is nef in the usual sense for 
every $w\in W$. 

When $A^1(X/S; W)$ is finitely generated, equivalently, 
$\dim N^1(X/S; W)$ is finite, we can formulate 
Kleiman's ampleness criterion (see 
\cite[Theorem 4.4.5]{fujino-cone-contraction}) 
and discuss the 
cone and contraction theorem for projective morphisms 
between complex analytic spaces (see 
Section \ref{a-sec9} below). 
We note that $A^1(X/S; W)$ is not always 
finitely generated. 

\begin{rem}[Nakayama's finiteness]\label{a-rem2.7}
By Nakayama's finiteness (see 
\cite[Subsection 4.1]{fujino-minimal} 
and \cite[Section 4.5]{fujino-cone-contraction}), 
it is known that $\dim N^1(X/S; W)$ is finite 
under the assumption that 
$W\cap Z$ has only finitely many connected components 
for any analytic subset $Z$ defined over an open neighborhood 
of $W$. 
In particular, if $W$ is a Stein compact subset of $Y$ such that 
$\Gamma (W, \mathcal O_Y)$ is noetherian, then 
$\dim N^1(X/S; W)$ is finite. 
\end{rem}
\end{say}

\begin{say}[Big $\mathbb R$-line bundles]\label{a-say2.8}
Let $\pi\colon X\to S$ be a projective 
morphism of complex analytic spaces such that 
$X$ is irreducible and let $\mathcal L$ 
be an $\mathbb R$-line bundle on $X$. 
If $\mathcal L$ is a finite positive $\mathbb R$-linear 
combination of $\pi$-big line bundles on $X$, then 
$\mathcal L$ is said to be {\em{big over $S$}}. 
\end{say}

We will use the following convention throughout this paper. 

\begin{say}\label{a-say2.9}
The expression \lq ... for every $m\gg 0$\rq \ means 
that \lq there exists a positive real number $m_0$ such that ... 
for every $m\geq m_0$.\rq
\end{say}

\section{On vanishing theorems}\label{a-sec3}

In this section, we will briefly recall the 
the vanishing theorems and the strict support condition 
established in 
\cite{fujino-analytic-vanishing}, which is an analytic generalization 
of \cite[Chapter 5]{fujino-foundations}. 
The reader can find all the details in \cite{fujino-analytic-vanishing} 
(see also \cite{fujino-fujisawa}, \cite{fujino-vanishing-pja}, 
and \cite[Chapter 3]{fujino-cone-contraction}). 
For a completely different approach due to Murayama, see \cite{murayama}. 
Let us start with the definition of 
{\em{analytic simple normal crossing pairs}}. 

\begin{defn}[Analytic simple normal crossing pairs]\label{a-def3.1}
Let $X$ be a simple normal crossing divisor 
on a smooth complex analytic space $M$ and 
let $B$ be an $\mathbb R$-divisor on $M$ such that 
$\Supp (B+X)$ is a simple normal crossing divisor on $M$ and 
that $B$ and $X$ have no common irreducible components. 
Then we put $D:=B|_X$ and 
consider the pair $(X, D)$. 
We call $(X, D)$ an {\em{analytic globally embedded simple 
normal crossing pair}} and $M$ the {\em{ambient space}} 
of $(X, D)$. 

If the pair $(X, D)$ is locally isomorphic to an analytic 
globally embedded 
simple normal crossing pair at any point of $X$ and the irreducible 
components of $X$ and $D$ are all smooth, 
then $(X, D)$ is called an {\em{analytic simple normal crossing 
pair}}. 

As we explained 
in \ref{a-say2.4}, we use the symbol $K_X$ as a formal divisor 
class with an isomorphism $\mathcal O_X(K_X)\simeq \omega_X$ if 
there is no danger of confusion, where 
$\omega_X$ is the {\em{dualizing sheaf}} of $X$. 
\end{defn}

\begin{rem}\label{a-rem3.2}
Let $X$ be a smooth complex analytic space and let $D$ be 
an $\mathbb R$-divisor on $X$ such that $\Supp D$ is a simple 
normal crossing divisor on $X$. 
Then, by considering $M:=X\times \mathbb C$, 
we can see $(X, D)$ as an analytic globally embedded simple 
normal crossing pair. 
\end{rem}
The notion of {\em{strata}}, 
which is a generalization of that of 
log canonical centers, plays a crucial role. 

\begin{defn}[Strata]\label{a-def3.3}
Let $(X, D)$ be an analytic simple normal crossing pair 
such that $D$ is effective. 
Let $\nu\colon X^\nu\to X$ be the normalization. 
We put 
\begin{equation*} 
K_{X^\nu}+\Theta:=\nu^*(K_X+D). 
\end{equation*}  
This means that $\Theta$ is the union of $\nu^{-1}_*D$ and the 
inverse image of the singular locus of $X$. 
If $S$ is an irreducible component of $X$ 
or the $\nu$-image 
of some log canonical center of $(X^\nu, \Theta)$, 
then $S$ is called a {\em{stratum}} of $(X, D)$. 
By definition, $S$ is a stratum of $(X, D)$ if and 
only if $S$ is a stratum of $(X, D^{=1})$. 
\end{defn}

We recall Siu's theorem on coherent analytic sheaves, 
which is a special case of \cite[Theorem 4]{siu}. 

\begin{thm}\label{a-thm3.4} 
Let $\mathcal F$ be a coherent sheaf on a complex 
analytic space $X$. 
Then there exists a locally finite family $\{Y_i\}_{i\in I}$ 
of complex analytic subvarieties of $X$ such that 
\begin{equation*}
\Ass _{\mathcal O_{X,x}}(\mathcal F_x)=\{\mathfrak{p}_{x, 1}, 
\ldots, \mathfrak{p}_{x, r(x)}\}
\end{equation*}
holds for every point $x\in X$, where 
$\mathfrak{p}_{x, 1}, 
\ldots, \mathfrak{p}_{x, r(x)}$ are the prime ideals 
of $\mathcal O_{X, x}$ associated to the irreducible components 
of the germs $Y_{i, x}$ of $Y_i$ at $x$ with $x\in Y_i$. 
We note that each $Y_i$ is called an {\em{associated subvariety}} 
of $\mathcal F$. 
\end{thm}
Now we are ready to state the main result of \cite{fujino-analytic-vanishing}. 

\begin{thm}[{\cite[Theorem 1.1]{fujino-analytic-vanishing}}]\label{a-thm3.5}
Let $(X, \Delta)$ be an analytic simple 
normal crossing pair such that $\Delta$ is a boundary 
$\mathbb R$-divisor on $X$. 
Let $f\colon X\to Y$ be a projective morphism 
to a complex analytic space $Y$ and let $\mathcal L$ 
be a line bundle on $X$. 
Let $q$ be an arbitrary non-negative integer. 
Then we have the following properties. 
\begin{itemize}
\item[(i)] $($Strict support condition$)$.~If 
$\mathcal L-(\omega_X+\Delta)$ is $f$-semi-ample,  
then every 
associated subvariety of $R^qf_*\mathcal L$ is the $f$-image 
of some stratum of $(X, \Delta)$. 
\item[(ii)] $($Vanishing theorem$)$.~If 
$\mathcal L-(\omega_X+\Delta)\sim _{\mathbb R} f^*\mathcal H
$ holds 
for some $\pi$-ample 
$\mathbb R$-line bundle $\mathcal H$ on $Y$, where 
$\pi\colon Y\to Z$ is a 
projective morphism to a complex analytic space 
$Z$, then we have 
\begin{equation*}
R^p\pi_*R^qf_*\mathcal L=0
\end{equation*}
for every $p>0$. 
\end{itemize} 
\end{thm}

We make a supplementary remark on Theorem \ref{a-thm3.5}. 

\begin{rem}\label{a-rem3.6} 
In Theorem \ref{a-thm3.5}, 
we always assume that $\Delta$ is {\em{globally $\mathbb R$-Cartier}}, 
that is, $\Delta$ is a finite $\mathbb R$-linear combination 
of Cartier divisors. 
Under this assumption, we can obtain an $\mathbb R$-line bundle 
$\mathcal N$ 
on $X$ naturally associated to $\mathcal L-(\omega_X+\Delta)$. 
The assumption in (i) means that $\mathcal N$ is a finite 
positive $\mathbb R$-linear combination of $\pi$-semi-ample 
line bundles on $X$. 
The assumption in (ii) says that $\mathcal N=f^*\mathcal H$ holds 
in $\Pic(X)\otimes _{\mathbb Z} \mathbb R$. 
\end{rem}

We do not prove Theorem \ref{a-thm3.5} here. 
For the details of the proof of Theorem \ref{a-thm3.5}, see 
\cite{fujino-analytic-vanishing}, which depends on 
Saito's theory of mixed Hodge modules (see 
\cite{saito1}, \cite{saito2}, \cite{saito3}, 
\cite{fujino-fujisawa-saito}, 
and \cite{saito4}) and 
Takegoshi's analytic generalization of Koll\'ar's torsion-free and 
vanishing theorem (see \cite{takegoshi}). 
In \cite{fujino-fujisawa}, the reader can find an 
alternative approach to 
Theorem \ref{a-thm3.5} without using Saito's theory of mixed 
Hodge modules (see also \cite{murayama}). 
We note that 
Theorem \ref{a-thm3.5} is one of the main ingredients of 
this paper. Or, we can see this paper as an application of 
Theorem \ref{a-thm3.5}. 
In order to explain the vanishing theorem 
of Reid--Fukuda type, we prepare the notion of 
nef and log big $\mathbb R$-line bundles. 

\begin{defn}\label{a-def3.7}
Let $f\colon X\to Y$ and $\pi\colon Y\to Z$ 
be projective morphisms between complex analytic 
spaces and let $\mathcal H$ be an 
$\mathbb R$-line bundle on $Y$. 
Let $\Delta$ be a boundary $\mathbb R$-divisor 
on $X$ such that $(X, \Delta)$ is an 
analytic simple normal crossing pair. 
We say that $\mathcal H$ is {\em{nef and 
log big over $Z$ with respect to $f\colon (X, \Delta)
\to Y$}} if $\mathcal H$ is nef over $Z$ and $\mathcal H|_{f(S)}$ 
is big over $\pi\circ f(S)$ for every 
stratum $S$ of $(X, \Delta)$. 
\end{defn}

We note that if $\mathcal H$ is $\pi$-ample then 
it is nef and log big over $Z$ with respect to $f\colon (X, \Delta)\to Y$. 
Therefore, Theorem \ref{a-thm3.8} is obviously 
a generalization of 
Theorem \ref{a-thm3.5} (ii). 

\begin{thm}[{Vanishing theorem of Reid--Fukuda type, 
see \cite[Theorem 1.2]{fujino-analytic-vanishing}}]\label{a-thm3.8}
Let $(X, \Delta)$ be an analytic simple 
normal crossing pair such that $\Delta$ is a boundary 
$\mathbb R$-divisor on $X$. 
Let $f\colon X\to Y$ and $\pi\colon Y\to Z$ be projective morphisms 
between complex analytic spaces and let $\mathcal L$ 
be a line bundle on $X$. 
If $\mathcal L-(\omega_X+\Delta)\sim _{\mathbb R} f^*\mathcal H$ 
holds such that $\mathcal H$ is an $\mathbb R$-line bundle, 
which is nef and 
log big over $Z$ with respect to $f\colon (X, \Delta)\to Y$, on $Y$, then  
\begin{equation*}
R^p\pi_*R^qf_*\mathcal L=0
\end{equation*} 
holds for every $p>0$ and every $q$. 
\end{thm}

The reader can find the detailed proof of 
Theorem \ref{a-thm3.8} in \cite{fujino-analytic-vanishing}, 
which is harder than that of Theorem \ref{a-thm3.5} (ii). 
As an easy application of Theorem \ref{a-thm3.8}, 
we can establish the vanishing theorem of Reid--Fukuda type 
of log canonical pairs for projective morphisms between 
complex analytic spaces. 
Theorem \ref{a-thm3.9} can be seen as a 
generalization of the Kawamata--Viehweg vanishing theorem 
for projective morphisms between complex analytic spaces. 

\begin{thm}[Vanishing theorem of Reid--Fukuda type for 
log canonical pairs]\label{a-thm3.9}
Let $(X, \Delta)$ be a log canonical pair and let $\pi\colon X\to Y$ 
be a projective morphism of complex analytic spaces. 
Let $L$ be a $\mathbb Q$-Cartier integral Weil divisor on $X$. 
Assume that $L-(K_X+\Delta)$ is nef and 
big over $Y$ and 
that 
$(L-(K_X+\Delta))|_C$ is big 
over $\pi(C)$ for 
every log canonical center $C$ of $(X, \Delta)$. 
Then 
\begin{equation*}
R^q\pi_*\mathcal O_X(L)=0
\end{equation*} 
holds for every $q>0$. 
\end{thm}
\begin{proof}
The proof of \cite[Theorem 5.7.6]{fujino-foundations} works 
by Theorem \ref{a-thm3.8}. 
\end{proof}

Theorem \ref{a-thm3.9} will be generalized for 
semi-log canonical pairs in Theorem \ref{a-thm10.2} 
by using Theorem \ref{a-thm1.7}. 

\section{Quasi-log structures for complex analytic spaces}\label{a-sec4}

This section is the main part of this paper. 
In this section, we will discuss {\em{quasi-log structures}} on 
complex analytic spaces. 
For the details of the theory of quasi-log schemes, 
see \cite[Chapter 6]{fujino-foundations} and 
\cite{fujino-quasi}. 

Let us define {\em{quasi-log complex analytic spaces}}. 

\begin{defn}[Quasi-log complex analytic spaces]
\label{a-def4.1}
A {\em{quasi-log complex analytic space}} 
\begin{equation*}
\bigl(X, \omega, f\colon (Y, B_Y)\to X\bigr)
\end{equation*}
is a complex analytic space $X$ endowed with an 
$\mathbb R$-line bundle (or a globally $\mathbb R$-Cartier divisor) 
$\omega$ on $X$, a closed analytic subspace 
$X_{-\infty}\subsetneq X$, and a finite collection $\{C\}$ of 
reduced 
and irreducible closed analytic subspaces of $X$ such that there 
exists a 
projective morphism $f\colon (Y, B_Y)\to X$ from an 
analytic globally 
embedded simple 
normal crossing pair $(Y, B_Y)$ 
satisfying the following properties: 
\begin{itemize}
\item[(1)] $f^*\omega\sim_{\mathbb R}K_Y+B_Y$. 
\item[(2)] The natural map 
$\mathcal O_X
\to f_*\mathcal O_Y(\lceil -(B_Y^{<1})\rceil)$ 
induces an isomorphism 
\begin{equation*}
\mathcal I_{X_{-\infty}}\overset{\simeq}
{\longrightarrow} f_*\mathcal O_Y(\lceil 
-(B_Y^{<1})\rceil-\lfloor B_Y^{>1}\rfloor),  
\end{equation*}
where $\mathcal I_{X_{-\infty}}$ is the defining ideal sheaf of 
$X_{-\infty}$ on $X$. 
\item[(3)] The collection of 
closed analytic subvarieties $\{C\}$ coincides with the $f$-images 
of $(Y, B_Y)$-strata that are not included in $X_{-\infty}$. 
\end{itemize}
We often simply write $[X, \omega]$ to denote 
the above data 
\begin{equation*}
\bigl(X, \omega, f\colon (Y, B_Y)\to X\bigr)
\end{equation*}
if there is no risk of confusion. 
The closed analytic subvarieties $C$ are called 
the {\em{qlc strata}} of $[X, \omega]$. 
If a qlc stratum $C$ is not an irreducible component of $X$, 
then it is called a {\em{qlc center}} of $[X, \omega]$. 
The closed analytic subspace $X_{-\infty}$ is called 
the {\em{non-qlc locus}} of $[X, \omega]$. 
We note that we sometimes use $\Nqlc(X, \omega)$ 
or $\Nqlc\left(X, \omega, f\colon (Y, B_Y)\to X\right)$ to 
denote $X_{-\infty}$. 
We usually call $f\colon (Y, B_Y)\to X$ a 
{\em{quasi-log resolution}} of $[X, \omega]$. 

In the above definition, if $\omega$ is a $\mathbb Q$-line bundle 
(or a globally $\mathbb Q$-Cartier divisor), 
$B_Y$ is a $\mathbb Q$-divisor, 
and $f^*\omega\sim _{\mathbb Q} K_Y+B_Y$ holds, 
then we say that 
\begin{equation*}
\left(X, \omega, f\colon (Y, B_Y)\to X\right)
\end{equation*} 
has a {\em{$\mathbb Q$-structure}}. 
\end{defn}

We make an important remark. 

\begin{rem}\label{a-rem4.2} 
As in Remark \ref{a-rem1.2}, we can naturally 
see $\omega$ as 
an $\mathbb R$-line bundle on $X$ in Definition 
\ref{a-def4.1}. 
In Definition \ref{a-def4.1} (1), $f^*\omega\sim _{\mathbb R} K_Y+B_Y$ means 
that $B_Y$ is globally $\mathbb R$-Cartier, that is, 
$B_Y$ is a finite $\mathbb R$-linear combination 
of Cartier divisors, and that $f^*\omega=\omega_Y+\mathcal B_Y$ holds in 
$\Pic(Y)\otimes _{\mathbb Z}\mathbb R$, where 
$\omega_Y$ is the dualizing sheaf of $Y$ and $\mathcal B_Y$ 
is an 
$\mathbb R$-line bundle on 
$Y$ naturally associated to 
the globally $\mathbb R$-Cartier divisor $B_Y$. 
Similarly, $f^*\omega\sim _{\mathbb Q} K_Y+B_Y$ means 
that $f^*\omega=\omega_Y+\mathcal B_Y$ holds 
in $\Pic(Y)\otimes _{\mathbb Z}\mathbb Q$. 
\end{rem}

The notion of {\em{quasi-log canonical pairs}} is useful. 

\begin{defn}[Quasi-log canonical pairs]\label{a-def4.3}
In Definition \ref{a-def4.1}, if $X_{-\infty}=\emptyset$, then 
\begin{equation*}
\left(X, \omega, f\colon (Y, B_Y)\to X\right)
\end{equation*} 
is called 
a {\em{quasi-log canonical pair}}. 
We sometimes simply say that $[X, \omega]$ is a {\em{qlc pair}}. 
\end{defn}

The most important result on quasi-log complex analytic spaces 
is the following adjunction formula. 
It is an easy consequence of the strict support condition 
in Theorem \ref{a-thm3.5} (i). 

\begin{thm}[Adjunction formula for quasi-log complex 
analytic spaces]\label{a-thm4.4} 
Let 
\begin{equation*}
\bigl(X, \omega, f\colon (Y, B_Y)\to X\bigr)
\end{equation*} 
be a quasi-log complex analytic space 
and let $X'$ be the union of 
$X_{-\infty}$ with a union of some 
qlc strata of $[X, \omega]$. 
Then, 
after replacing $X$ with any relatively compact open subset of $X$,  
we can construct a projective morphism $f'\colon (Y', B_{Y'})\to X'$ 
from an analytic globally embedded simple normal crossing 
pair $(Y', B_{Y'})$ such that 
\begin{equation*}
\bigl(X', \omega', f'\colon (Y', B_{Y'})\to X'\bigr)
\end{equation*} 
is a quasi-log complex analytic space with $\omega'=\omega|_{X'}$ and 
$X'_{-\infty}=X_{-\infty}$. Moreover, the qlc strata of $[X', \omega']$ are 
exactly the qlc strata of $[X, \omega]$ that are included in $X'$. 
\end{thm}

The proof of \cite[Theorem 6.3.5 (i)]{fujino-foundations} works 
without any modifications. 

\begin{proof}[Sketch of Proof of Theorem \ref{a-thm4.4}] 
We replace $X$ with a relatively compact open subset of $X$. 
Let $M$ be the ambient space of $(Y, B_Y)$. 
By taking a suitable projective bimeromorphic 
modification of 
$M$ (see \cite[Proposition 6.3.1]{fujino-foundations}), 
we may assume that the union of all strata of $(Y, B_Y)$ mapped to $X'$, 
which is denoted by $Y'$, is a union of some irreducible components 
of $Y$. We put 
$K_{Y'}+B_{Y'}:=(K_Y+B_Y)|_{Y'}$ and $Y'':=Y-Y'$. 
We also put $A:=\lceil -(B^{<1}_Y)\rceil$ and 
$N:=\lfloor B^{>1}_Y\rfloor$, and 
consider the following short exact sequence: 
\begin{equation*}
0\to \mathcal O_{Y''}(A-N-Y')\to \mathcal O_Y(A-N)\to 
\mathcal O_{Y'}(A-N)\to 0. 
\end{equation*} 
Then we have the following long exact sequence: 
\begin{equation}
\begin{split}\label{a-eq4.1}
0&\longrightarrow f_*\mathcal O_{Y''}(A-N-Y')\longrightarrow 
f_*\mathcal O_Y(A-N)
\longrightarrow 
f_*\mathcal O_{Y'}(A-N)\\ 
& \overset{\delta}{\longrightarrow} 
R^1f_*\mathcal O_{Y''}(A-N-Y')\longrightarrow 
\cdots . 
\end{split}
\end{equation}
By the strict support condition in Theorem \ref{a-thm3.5} (i), 
every associated subvariety of 
\begin{equation*}
R^1f_*\mathcal O_{Y''}(A-N-Y')
\end{equation*} 
is the $f$-image of some stratum of 
$(Y'', \{B_{Y''}\}+B^{=1}_{Y''}-Y'|_{Y''})$, where 
$K_{Y''}+B_{Y''}:=(K_Y+B_Y)|_{Y''}$, since 
\begin{equation*}
\begin{split}
(A-N-Y')|_{Y''}-(K_{Y''}+\{B_{Y''}\}+B^{=1}_{Y''}-Y'|_{Y''}) 
&= -(K_{Y''}+B_{Y''})\\ 
&\sim _{\mathbb R} -(f^*\omega)|_{Y''}. 
\end{split} 
\end{equation*} 
On the other hand, the support of $f_*\mathcal O_{Y'}(A-N)$ is contained 
in $f(Y')$. 
Hence, the connecting homomorphism $\delta$ in 
\eqref{a-eq4.1} is zero. 
Thus we obtain the following short exact sequence 
\begin{equation}\label{a-eq4.2}
0\to f_*\mathcal O_{Y''}(A-N-Y')\to \mathcal I_{X_{-\infty}}
\to f_*\mathcal O_{Y'}(A-N)\to 0. 
\end{equation} 
We put $\mathcal I_{X'}:=f_*\mathcal O_{Y''}(A-N-Y')$ and 
define a complex analytic space structure on $X'$ by $\mathcal I_{X'}$. 
Then we can check that $f':=f|_{Y'}\colon (Y', B_{Y'})\to 
X'$ and $\omega':=\omega|_{X'}$ satisfy all the desired properties. 
We note that 
the short exact sequence \eqref{a-eq4.2} is 
\begin{equation}\label{a-eq4.3}
0\to \mathcal I_{X'}\to \mathcal I_{X_{-\infty}}\to \mathcal I_{X'_{-\infty}}\to 0. 
\end{equation}
For the details, see, for example, the proof of 
\cite[Theorem 6.3.5 (i)]{fujino-foundations}. 
\end{proof}

The following theorem is an important supplement to 
Theorem \ref{a-thm4.4}. 

\begin{thm}\label{a-thm4.5}
The complex analytic structure on $X'$ of the 
quasi-log complex analytic space 
\begin{equation*}
\bigl(X', \omega', f'\colon (Y', B_{Y'})\to X'\bigr)
\end{equation*} 
defined in Theorem \ref{a-thm4.4} 
is independent of the construction of $f'\colon (Y', B_{Y'})\to X'$. 
Therefore, the defining ideal sheaf $\mathcal I_{X'}$ of $X'$ 
is a globally well-defined coherent ideal sheaf on $X$. 
\end{thm}

\begin{proof}[Sketch of Proof of Theorem \ref{a-thm4.5}]
On any relatively compact open subset $U$ of $X$, 
we defined $\mathcal I_{X'}$ in the proof of 
Theorem \ref{a-thm4.4}. 
By \cite[Proposition 6.3.6]{fujino-foundations}, 
we see that it is independent of the construction. 
Hence we get a globally well-defined 
coherent defining ideal sheaf $\mathcal I_{X'}$ of $X'$. 
This is what we wanted. 
For the details, see the proof of \cite[Proposition 6.3.6]{fujino-foundations}.  
\end{proof}

For various inductive treatments, 
the notion of $\Nqklt(X, \omega)$ is very useful. 

\begin{cor}\label{a-cor4.6}
Let 
\begin{equation*}
\left(X, \omega, f\colon (Y, B_Y)\to X\right)
\end{equation*} 
be a quasi-log complex analytic space. 
The union of $X_{-\infty}$ with all qlc centers of $[X, \omega]$ is 
denoted by $\Nqklt(X, \omega)$, or, more precisely, 
\begin{equation*}
\Nqklt\left(X, \omega, f\colon (Y, B_Y)\to X\right). 
\end{equation*} 
If $\Nqklt(X, \omega)\ne X_{-\infty}$, then, after 
replacing $X$ with any relatively compact open subset of $X$, 
\begin{equation*}
\left[ \Nqklt(X, \omega), \omega|_{\Nqklt(X, \omega)}\right]
\end{equation*} 
naturally becomes a quasi-log complex analytic space by adjunction. 
\end{cor}

\begin{proof}
This is a special case of Theorem \ref{a-thm4.4}. 
\end{proof}

If we apply Corollary \ref{a-cor4.6} to Example \ref{a-ex1.5}, 
then $\Nqklt(X, K_X+\Delta)=\Nklt(X, \Delta)$ holds, 
where $\Nklt(X, \Delta)$ denotes the {\em{non-klt locus}} 
of $(X, \Delta)$. 
Moreover, we have $\mathcal I_{\Nqklt(X, K_X+\Delta)}=\mathcal J(X, \Delta)$, 
where $\mathcal J(X, \Delta)$ is the usual 
{\em{multiplier ideal sheaf}} of $(X, \Delta)$ and $\mathcal I_{\Nqklt(X, K_X+
\Delta)}$ is the defining ideal sheaf of $\Nqklt(X, K_X+\Delta)$ on $X$. 

For geometric applications, we need vanishing theorems. 
Of course, they follow from the vanishing theorems 
for analytic simple normal crossing pairs (see 
Theorem \ref{a-thm3.5} (ii) and Theorem \ref{a-thm3.8}). 

\begin{thm}[Vanishing theorem I]\label{a-thm4.7}
Let 
\begin{equation*}
\bigl(X, \omega, f\colon (Y, B_Y)\to X\bigr)
\end{equation*} 
be a quasi-log complex analytic space 
and let $\pi\colon X\to S$ be a projective morphism between 
complex analytic spaces. 
Let $\mathcal L$ be a line bundle on $X$ such that 
$\mathcal L-\omega$ is nef and log big over $S$ with 
respect to $[X, \omega]$, that is, $\mathcal L-\omega$ is nef 
over $S$ and $(\mathcal L-\omega)|_C$ is big over 
$\pi(C)$ for every qlc stratum $C$ of $[X, \omega]$. 
Then 
\begin{equation*}
R^i\pi_*(\mathcal I_{X_{-\infty}}
\otimes \mathcal L)=0
\end{equation*} 
holds for every $i>0$. 
\end{thm}

\begin{proof}[Sketch of Proof of Theorem \ref{a-thm4.7}]
We take an arbitrary point $s\in S$. 
It is sufficient to prove 
$R^i\pi_*(\mathcal I_{X_{-\infty}}
\otimes \mathcal L)=0$ for every $i>0$ on a relatively compact 
open neighborhood $U_s$ of $s\in S$. 
We replace $X$ and $S$ with $\pi^{-1}(U_s)$ and 
$U_s$, respectively. 
From now on, we use the notation in the proof of 
Theorem \ref{a-thm4.4}. 
In the proof of Theorem \ref{a-thm4.4}, 
we put $X':=X_{-\infty}$. 
Then $Y'$ is the union of all strata of $(Y, B_Y)$ mapped to 
$X_{-\infty}$. 
In this situation, 
we can check that the following natural inclusion  
\begin{equation*} 
f_*\mathcal O_{Y''}(A-N-Y')\hookrightarrow f_*\mathcal O_Y(A-N)
\end{equation*} 
is an isomorphism. 
Since 
\begin{equation*}
\left(f^*\mathcal L+(A-N-Y')\right)|_{Y''}-(K_{Y''}+\{B_{Y''}\}+B^{=1}_{Y''}-Y'|_{Y''})
\sim _{\mathbb R} (f^*(\mathcal L-\omega))|_{Y''}, 
\end{equation*} 
we obtain 
\begin{equation*}
\begin{split}
R^i\pi_*(\mathcal L\otimes \mathcal I_{X_{-\infty}}) 
&\simeq R^i\pi_*\left(\mathcal L\otimes f_*\mathcal O_Y(A-N)\right)
\\ & =
R^i\pi_*\left(\mathcal L\otimes f_*\mathcal O_{Y''}(A-N-Y')\right)\\ 
&=R^i\pi_*\left(f_*(f^*\mathcal L\otimes \mathcal O_{Y''}(A-N-Y'))\right)=0
\end{split}
\end{equation*} 
for every $i>0$ by Theorem \ref{a-thm3.8}. 
\end{proof}

The following vanishing theorem and Theorem \ref{a-thm4.4} 
will play a crucial role in the theory of quasi-log complex analytic 
spaces. We can see it as a generalization of 
the Kawamata--Viehweg--Nadel vanishing theorem 
for projective morphisms between complex analytic spaces. 

\begin{thm}[Vanishing theorem II]\label{a-thm4.8}
Let 
\begin{equation*}
\bigl(X, \omega, f\colon (Y, B_Y)\to X\bigr)
\end{equation*} 
be a quasi-log complex analytic space 
and let $X'$ be the union of 
$X_{-\infty}$ with a union of some 
qlc strata of $[X, \omega]$. 
Let $\pi\colon X\to S$ be a projective morphism between 
complex analytic spaces and 
let $\mathcal L$ be a line bundle on $X$ such that 
$\mathcal L-\omega$ is nef over $S$ and $(\mathcal L-\omega)|_C$ is 
big over $\pi(C)$ for every qlc stratum $C$ of $[X, \omega]$ 
which is not contained in $X'$. 
Then 
\begin{equation*}
R^i\pi_*(\mathcal I_{X'}
\otimes \mathcal L)=0
\end{equation*} 
holds for every $i>0$, where $\mathcal I_{X'}$ is the 
defining ideal sheaf of $X'$ on $X$. 
In particular, if $\mathcal L-\omega$ is ample over $S$, 
then 
\begin{equation*}
R^i\pi_*(\mathcal I_{X'}
\otimes \mathcal L)=0
\end{equation*} 
holds for every $i>0$. 
\end{thm}
\begin{proof}[Sketch of Proof of Theorem \ref{a-thm4.8}]
We take an arbitrary point $s\in S$. 
It is sufficient to prove $R^i\pi_*(\mathcal I_{X'}\otimes 
\mathcal L)=0$ for every $i>0$ on a relatively compact open neighborhood 
$U_s$ of $s\in S$. 
Therefore, we replace $S$ and $X$ with 
$U_s$ and $\pi^{-1}(U_s)$, respectively. 
From now on, we use the same notation as in the proof of 
Theorem \ref{a-thm4.4}. 
We note that 
\begin{equation*}
\left(f^*\mathcal L+(A-N-Y')\right)|_{Y''}-(K_{Y''}+\{B_{Y''}\}+B^{=1}_{Y''}-Y'|_{Y''})
\sim _{\mathbb R} (f^*(\mathcal L-\omega))|_{Y''}
\end{equation*} 
holds. By Theorem \ref{a-thm3.8}, 
we obtain 
\begin{equation*}
\begin{split}
R^i\pi_*(\mathcal L\otimes \mathcal I_{X'}) 
&= R^i\pi_*\left(\mathcal L\otimes f_*\mathcal O_{Y''}(A-N-Y')\right)\\ 
&= R^i\pi_*\left(f_*(f^*\mathcal L\otimes \mathcal O_{Y''}(A-N-Y'))\right)=0
\end{split}
\end{equation*} 
for every $i>0$. 
We finish the proof. 
\end{proof}

As in the algebraic case, the following important property holds. 

\begin{lem}\label{a-lem4.9}
Let $[X, \omega]$ be a quasi-log complex analytic space with 
$X_{-\infty}=\emptyset$, that is, 
$[X, \omega]$ is a quasi-log canonical pair. 
We assume that every qlc stratum of $[X, \omega]$ 
is an irreducible component of $X$, equivalently, $\Nqklt(X, \omega)=
\emptyset$. 
Then $X$ is normal. 
\end{lem}

\begin{proof}[Sketch of Proof of Lemma \ref{a-lem4.9}]
Let $f\colon (Y, B_Y)\to X$ be a quasi-log resolution. 
By the assumption that $X_{-\infty}$ is empty, 
we see that the natural map 
\begin{equation*}
\mathcal O_X\to f_*\mathcal O_Y(\lceil -(B^{<1}_Y)\rceil)
\end{equation*} 
is an isomorphism. 
This implies that $\mathcal O_X\simeq f_*\mathcal O_Y$ holds. 
Hence $f$ has connected fibers. 
Therefore, every connected component of $X$ is an 
irreducible component of $X$ by the assumption that 
every stratum of $[X, \omega]$ is an irreducible component of $X$. 
Thus, we may assume that $X$ is irreducible and every stratum of $Y$  
is mapped onto $X$. In this case, it is well known and 
is easy to prove that $X$ is normal. 
For the details, see, 
for example, 
the proof of \cite[Theorem 3.4.1]{fujino-cone-contraction}. 
\end{proof}

By Theorem \ref{a-thm4.4} and Lemma \ref{a-lem4.9}, 
we can prove:  

\begin{thm}[Basic properties of qlc strata]\label{a-thm4.10}
Let $[X, \omega]$ be a quasi-log complex analytic space with 
$X_{-\infty}=\emptyset$, 
that is, $[X, \omega]$ is a quasi-log canonical pair. 
Then we have the following properties. 
\begin{itemize}
\item[(i)] The intersection of two qlc strata is a union 
of some qlc strata. 
\item[(ii)] For any point $x\in X$, 
the set of all qlc strata passing through 
$x$ has a unique minimal {\em{(}}with respect to the inclusion{\em{)}} 
element $C_x$. Moreover, 
$C_x$ is normal at $x$. 
\end{itemize}
\end{thm}

\begin{proof}[Sketch of Proof of Theorem \ref{a-thm4.10}]
Let $C_1$ and $C_2$ be two qlc strata of $[X, \omega]$. 
We may assume that $C_1\ne C_2$ with $C_1\cap C_2\ne 
\emptyset$. 
We take $P\in C_1\cap C_2$. 
It is sufficient to find a qlc stratum $C$ such that $P\in 
C\subset C_1\cap C_2$ for the proof of (i). 
We put $X':=C_1\cup C_2$ and $\omega':=\omega|_{X'}$. 
Then, after shrinking $X$ around $P$ suitably, 
$[X', \omega']$ becomes a quasi-log complex analytic space 
by adjunction (see Theorem \ref{a-thm4.4}). 
Note that $X'$ is reducible at $P$. 
Therefore, by Lemma \ref{a-lem4.9} above, 
there exists a qlc center $C^\dag$ of $[X, \omega]$ 
such that $P\in C^\dag\subset X'$. 
By this fact, we can easily prove (i). 
For the details, see the proof of 
\cite[Theorem 6.3.11 (i)]{fujino-foundations}. 
The uniqueness of the minimal (with respect to 
the inclusion) qlc stratum follows from (i) and 
the normality of the minimal qlc stratum follows 
from Lemma \ref{a-lem4.9}. 
So we finish the proof of (ii). 
\end{proof}

Theorem \ref{a-thm4.10} will play a crucial role in the theory 
of quasi-log complex analytic spaces. 

\section{Some basic operations on 
quasi-log structures}\label{a-sec5}

In this section, we will discuss some basic operations on 
quasi-log structures in the complex analytic setting. 
Almost all of them are well known for quasi-log schemes 
(see \cite[Section 3]{fujino-reid-fukuda}, 
\cite[Chapter 6]{fujino-foundations}, 
\cite[Subsection 4.3]{fujino-cone}, 
and so on). 
They will play an important role in the subsequent sections. 

\begin{lem}\label{a-lem5.1}
Let 
\begin{equation*}
\left(X, \omega, f\colon (Y, B_Y)\to X\right)
\end{equation*} 
be a quasi-log complex analytic space and 
let $P\in X$ be a point. 
Let $D_1, \ldots, D_k$ be effective 
Cartier divisors on $X$ such that $P\in \Supp D_i$ for every $i$. 
We assume that no irreducible component of $Y$ is mapped 
into $\bigcup _{i=1}^k\Supp D_i$. 
Then, after replacing $X$ with any relatively compact 
open neighborhood of $P$, 
$\left[X, \omega+\sum _{i=1}^kD_i\right]$ naturally becomes a quasi-log 
complex analytic space. 
Moreover, if $\Nqlc\left(X, \omega+\sum _{i=1}^k D_i\right)
=\emptyset$, 
then $k\leq \dim _P X$ holds. 
More precisely, $k\leq \dim _P C_P$ holds, 
where $C_P$ is the minimal qlc stratum of $[X, \omega]$ 
passing through $P$. 
\end{lem}

\begin{proof}[Sketch of Proof of Lemma \ref{a-lem5.1}]
After replacing $X$ with a relatively compact open neighborhood 
of $P$, 
we may assume that 
\begin{equation*}
\left(Y, \sum _{i=1}^k f^*D_i +\Supp B_Y\right)
\end{equation*} 
is an analytic globally embedded simple normal crossing pair 
by \cite[Proposition 6.3.1]{fujino-foundations} and 
\cite{bierstone-milman2}. 
Then we see that 
\begin{equation*}
f\colon \left(Y, B_Y+\sum _{i=1}^k f^*D_i\right)\to X
\end{equation*} 
naturally 
gives a quasi-log structure on $\left[X, \omega+\sum _{i=1}^k D_i\right]$. 
When we prove $k\leq \dim _P X$, 
we can freely replace $X$ with a relatively compact open neighborhood 
of $P$. 
Hence, we can easily see that the proof of 
\cite[Lemma 6.3.13]{fujino-foundations} works 
with only some minor modifications. 
\end{proof}

Lemma \ref{a-lem5.2} is a very 
basic result in the theory of quasi-log complex analytic spaces. 

\begin{lem}\label{a-lem5.2} 
Let $[X, \omega]$ be a quasi-log complex analytic space 
and let $D$ be an effective $\mathbb R$-Cartier 
divisor on $X$. 
This means that $D$ is a finite positive 
$\mathbb R$-linear combination of effective Cartier divisors. 
Then, after replacing $X$ with any relatively compact open 
subset of $X$, $[X, \omega+D]$ naturally 
becomes a quasi-log complex analytic space. 
\end{lem}
\begin{proof}[Sketch of Proof of Lemma \ref{a-lem5.2}]
Let $f\colon (Y, B_Y)\to X$ be a quasi-log resolution. 
By replacing $X$ with any relatively compact open subset 
of $X$ and taking a suitable projective bimeromorphic 
modification of $M$, 
the ambient space of $(Y, B_Y)$, 
we may assume that the union of all strata of $(Y, B_Y)$ 
mapped to $\Supp D\cup \Nqlc(X, \omega)$ by $f$, 
which is denoted by $Y'$, is a union of some irreducible 
components of $Y$ (see \cite[Proposition 6.3.1]{fujino-foundations} and 
\cite{bierstone-milman2}). 
We put $Y'':=Y-Y'$, $K_{Y''}+B_{Y''}:=(K_Y+B_Y)|_{Y''}$, and 
$f'':=f|_{Y''}$. 
We may further assume that $\left(Y'', B_{Y''}+(f'')^*D\right)$ is an analytic 
globally embedded simple normal crossing pair. 
Then 
\begin{equation*}
\left(X, \omega+D, f''\colon (Y'', B_{Y''})\to X\right)
\end{equation*} 
naturally becomes a quasi-log complex analytic space. 
For the details, see the proof of \cite[Lemma 4.23]{fujino-cone}. 
We note that the quasi-log structure of 
$\left(X, \omega+D, f''\colon (Y'', B_{Y''})\to X\right)$ 
constructed above coincides with that of 
$\left(X, \omega, f\colon (Y, B_Y)\to X\right)$ outside 
$\Supp D$. 
\end{proof}

The following lemma is useful since we can 
reduce various problems to the case 
where $X$ is irreducible (see 
also \cite[Lemmas 4.19 and 4.20]{fujino-cone}). 

\begin{lem}[{see \cite[Lemmas 3.12 and 
3.14]{fujino-reid-fukuda}}]\label{a-lem5.3} 
Let $[X, \omega]$ be a quasi-log complex analytic space 
and 
let $\pi\colon X\to S$ be a projective morphism of 
complex analytic spaces. 
We put $X^\dag:=\overline {X\setminus X_{-\infty}}$, 
the closure of $X\setminus X_{-\infty}$ in $X$, 
with the reduced structure. 
Then, after replacing $S$ with any relatively compact 
open subset of $S$, 
$[X^\dag, \omega^\dag:=\omega|_{X^\dag}]$ naturally becomes 
a quasi-log complex analytic space with 
the following properties. 
\begin{itemize}
\item[(1)] $C$ is a qlc stratum of $[X, \omega]$ 
if and only if $C$ is a qlc stratum of $[X^\dag, \omega^\dag]$. 
\item[(2)] $\mathcal I_{\Nqlc(X^\dag, \omega^\dag)}=
\mathcal I_{\Nqlc(X, \omega)}$ holds. 
\end{itemize}
\end{lem}

\begin{proof}[Proof of Lemma \ref{a-lem5.3}]
In Step \ref{a-5.3-step1}, we will construct a quasi-log resolution 
\begin{equation*}
f'\colon (Y', B_{Y'})\to X
\end{equation*} 
such that 
every irreducible component of $Y'$ is mapped to $X^\dag$ by $f'$. 
Then, in Step \ref{a-5.3-step2}, we will construct the desired 
quasi-log structure on $[X^\dag, \omega^\dag]$. 
\setcounter{step}{0}
\begin{step}\label{a-5.3-step1}
Let $M$ be the ambient space of $(Y, B_Y)$. 
After replacing $S$ with any relatively compact open subset of $S$, 
by taking some projective bimeromorphic modification of $M$, we 
may assume that 
the union of all strata of $(Y, B_Y)$ that are not 
mapped 
to $\overline{X\setminus X_{-\infty}}$, 
which is denoted by $Y''$, 
is a union of some irreducible components of $Y$ 
(see \cite[Proposition 6.3.1]{fujino-foundations} and 
\cite{bierstone-milman2}). 
We may further assume that the union of 
all strata of $(Y, B_Y)$ mapped to 
$\overline {X\setminus X_{-\infty}}
\cap X_{-\infty}$ is a union of some 
irreducible components of $Y$. 
We put $Y':=Y-Y''$ and $K_{Y''}+B_{Y''}:=(K_Y+B_Y)|_{Y''}$. 
We consider the short exact sequence 
\begin{equation*}
0\to \mathcal O_{Y''}(-Y')\to \mathcal O_Y\to \mathcal O_{Y'}\to 0. 
\end{equation*} 
As usual, we put $A:=\lceil -(B_Y^{<1})\rceil$ and $N:=
\lfloor B_Y^{>1}\rfloor$. By 
applying $\otimes \mathcal O_Y(A-N)$, we have 
\begin{equation*}
0\to \mathcal O_{Y''}(A-N-Y')\to \mathcal O_Y(A-N)\to \mathcal O_{Y'}
(A-N)\to 0. 
\end{equation*}
By taking $R^if_*$, we obtain 
\begin{equation*}
\begin{split}
0&\to f_* \mathcal O_{Y''}(A-N-Y')
\to f_*\mathcal O_Y(A-N)\to f_* \mathcal O_{Y'}
(A-N)\\
&\to R^1f_*\mathcal O_{Y''}(A-N-Y')\to \cdots.
\end{split}
\end{equation*}
By the strict support condition (see 
Theorem \ref{a-thm3.5} (i)), 
no associated subvariety of 
\begin{equation*}
R^1f_*\mathcal O_{Y''}(A-N-Y')
\end{equation*} 
is contained 
in $f(Y')\cap X_{-\infty}$. 
Note that 
\begin{equation*}
\begin{split}
(A-N-Y')|_{Y''}-(K_{Y''}+\{B_{Y''}\}+B_{Y''}^{=1}
-Y'|_{Y''})
&=-(K_{Y''}+B_{Y''})
\\&\sim _{\mathbb R} -(f^*\omega)|_{Y''}. 
\end{split}
\end{equation*}
Therefore, the connecting homomorphism 
\begin{equation*}
\delta:f_*\mathcal O_{Y'}(A-N)\to R^1f_*\mathcal O_{Y''}(A-N-Y')
\end{equation*} 
is zero. 
This implies that 
\begin{equation*}
0\to f_*\mathcal O_{Y''}(A-N-Y')\to 
\mathcal I_{X_{-\infty}}\to f_*\mathcal O_{Y'}
(A-N)\to 0 
\end{equation*} 
is exact. 
The ideal sheaf $\mathcal J=f_*\mathcal O_{Y''}
(A-N-Y')$ is zero 
when it is restricted to $X_{-\infty}$ because 
$\mathcal J\subset \mathcal I_{X_{-\infty}}=\mathcal I_{\Nqlc(X, \omega)}$. 
On the other hand, $\mathcal J$ is zero on 
$X\setminus X_{-\infty}$ because 
$f(Y'')\subset X_{-\infty}$. 
Therefore, we obtain $\mathcal J=0$. 
Thus we have $\mathcal I_{X_{-\infty}}=f_*\mathcal O_{Y'}(A-N)$. 
So $f'=f|_{Y'}\colon(Y', B_{Y'})\to X$, where 
$K_{Y'}+B_{Y'}:=(K_Y+B_Y)|_{Y'}$, 
gives the same quasi-log structure as one 
given by $f\colon (Y, B_Y)\to X$. 
\end{step}
\begin{step}\label{a-5.3-step2}
Let $\mathcal I_{X^\dag}$ be the defining 
ideal sheaf of $X^\dag$ on $X$. 
Let $f':(Y', B_{Y'})\to X$ be the quasi-log resolution 
constructed in Step \ref{a-5.3-step1}. 
Note that 
\begin{align*}
\mathcal I_{X_{-\infty}}&=f'_*\mathcal O_{Y'}(A-N)\\
&=f'_*\mathcal O_{Y'}(-N)
\end{align*} 
and that 
\begin{equation*}
f'(N)=X_{-\infty}\cap f'(Y')=X_{-\infty}\cap X^\dag
\end{equation*} 
set theoretically, where 
$A=\lceil -(B_Y^{<1})\rceil$ and $N=\lfloor B_Y^{>1}\rfloor$. 
We note that $A|_{Y'}=\lceil -(B_{Y'}^{<1})\rceil$ and 
$N|_{Y'}=\lfloor B_{Y'}^{>1}\rfloor$ hold by 
definition. Moreover, we obtain 
\begin{equation*} 
\mathcal I_{X^\dag}\cap \mathcal I_{X_{-\infty}}
=\mathcal I_{X^\dag}\cap f_*\mathcal O_Y(A-N)\subset 
\mathcal J=f_*\mathcal O_{Y''}(A-N-Y')=\{0\}. 
\end{equation*} 
Thus we can construct the following big commutative diagram. 
\begin{equation*}
\xymatrix{&&0\ar[d]&0\ar[d]&
\\&&\mathcal I_{X_{-\infty}}\ar@{=}[r]\ar[d]& 
\mathcal I_{X^\dag_{-\infty}}\ar[d]&\\
0\ar[r]&\mathcal I_{X^\dag}\ar[r]\ar@{=}[d]&\mathcal O_X\ar[r]\ar[d]
&\mathcal O_{X^\dag}\ar[r]\ar[d]&0\\
0\ar[r]&\mathcal I_{X^\dag}\ar[r]&\mathcal O_{X_{-\infty}}\ar[r]\ar[d]
&\mathcal O_{X^\dag_
{-\infty}}\ar[d]\ar[r]&0\\
&&0&0&
}
\end{equation*}
By construction, $f'$ factors through $X^\dag$. 
We put $f^\dag\colon (Y', B_{Y'})\to X^\dag$. 
Then it is easy to see that 
$f^\dag
\colon (Y', B_{Y'})\to X^\dag$ gives the desired quasi-log structure 
on $[X^\dag, \omega^\dag]$. 
\end{step}
We finish the proof. 
\end{proof}

Although we do not need the following lemma in this paper, 
we state it here for the sake of completeness. 

\begin{lem}[{see \cite[Lemma 4.20]{fujino-cone}}]\label{a-lem5.4} 
In Lemma \ref{a-lem5.3}, 
we consider a set of some qlc strata $\{C_i\}_{i\in I}$ of $[X, \omega]$. 
We put 
\begin{equation*} 
\left(X^\dag\right)'=\Nqlc(X^\dag, \omega^\dag)\cup 
\left(\bigcup _{i\in I} C_i\right)
\end{equation*} 
and 
\begin{equation*}
X'=\Nqlc(X, \omega)\cup 
\left(\bigcup _{i\in I} C_i\right). 
\end{equation*}  
Then, after replacing $S$ with 
any relatively compact open 
subset of $S$, $\left[(X^\dag)', \omega^\dag|_{(X^\dag)'}\right]$ and 
$\left[X', \omega|_{X'}\right]$ 
naturally become quasi-log complex analytic spaces by adjunction and 
$\mathcal I_{(X^\dag)'}=\mathcal I_{X'}$ holds, where 
$\mathcal I_{(X^\dag)'}$ and $\mathcal I_{X'}$ are the 
defining ideal sheaves of $(X^\dag)'$ and $X'$ on 
$X^\dag$ and $X$, respectively. 
In particular, $\mathcal I_{\Nqklt(X^\dag, \omega^\dag)}
=\mathcal I_{\Nqklt(X, \omega)}$ holds. 
\end{lem}

\begin{proof}[Sketch of Proof of Lemma \ref{a-lem5.4}] 
Here, we will 
use the same notation as in the proof of Lemma \ref{a-lem5.3}. We know 
that $[(X^\dag)', \omega^{\dag}|_{(X^\dag)'}]$ and 
$[X', \omega|_{X'}]$ naturally become quasi-log complex analytic spaces  
by adjunction after replacing $S$ with any relatively compact open subset of $S$. 
Thus it is sufficient to prove the equality $\mathcal I_{(X^\dag)'}=
\mathcal I_{X'}$. As usual, by \cite[Proposition 6.3.1]{fujino-foundations} 
and \cite{bierstone-milman2}, 
we may further assume that the union of all strata of 
$(Y, B_Y)$ that are mapped to $X'$, which is denoted by 
$Z$, is a union of some irreducible components of $Y$. 
We note that $Z\geq Y''$. 
We put $Z'=Y-Z$. 
Then it is obvious that $Z'\leq Y'$ holds. 
By the proof of adjunction (see the 
proof of Theorem \ref{a-thm4.4} and 
\cite[Theorem 6.3.5 (i)]{fujino-foundations}), 
we see that 
\begin{equation*} 
\mathcal I_{(X^\dag)'}=f'_*\mathcal O_{Z'}(A-N-(Z-Y'')|_{Z'})
=f'_*\mathcal O_{Z'}(A-N-Z|_{Z'})
=\mathcal I_{X'}
\end{equation*}  
holds. We finish the proof. 
\end{proof}

Lemma \ref{a-lem5.5} will play a crucial role in the proof of 
the basepoint-free theorem of Reid--Fukuda type (see 
Theorem \ref{a-thm7.1}). 

\begin{lem}[{see \cite[Lemma 3.15]{fujino-reid-fukuda}}]\label{a-lem5.5}
Let $[X, \omega]$ be a quasi-log complex analytic space 
and let $\pi\colon X\to S$ be a projective 
morphism of complex analytic spaces. 
Let $E$ be an effective $\mathbb R$-Cartier 
divisor on $X$. 
This means that $E$ is a finite positive 
$\mathbb R$-linear combination of effective Cartier divisors.  
We put 
\begin{equation*}
\widetilde{\omega}:=\omega+\varepsilon E
\end{equation*} 
with $0<\varepsilon \ll 1$. 
Then, after replacing $S$ with 
any relatively compact open subset of $S$, 
$[X, \widetilde \omega]$ naturally becomes a 
quasi-log complex analytic space with the following 
properties. 
\begin{itemize}
\item[(1)] Let $\{C_i\}_{i\in I}$ be the set of 
all qlc centers of $[X, \omega]$ contained in $\Supp E$. 
We put 
\begin{equation*}
X^{\star}:=\left(\bigcup _{i\in I}C_i\right) 
\cup \Nqlc(X, \omega). 
\end{equation*} 
Then, by adjunction, 
\begin{equation*}
[X^{\star}, \omega^\star:=\omega|_{X^\star}]
\end{equation*} 
is a quasi-log complex analytic space and 
\begin{equation*}
X^\star=\Nqlc(X, \widetilde \omega)
\end{equation*} 
holds. More precisely, 
\begin{equation*}
\mathcal I_{X^\star}=\mathcal I_{\Nqlc(X, \widetilde \omega)}
\end{equation*} 
holds, where $\mathcal I_{X^\star}$ is the defining ideal 
sheaf of $X^\star$ on $X$. 
\item[(2)] $C$ is a qlc center of $[X, \widetilde \omega]$ 
if and only if $C$ is a qlc center of $[X, \omega]$ 
with $C\not\subset \Supp E$. 
\end{itemize}
\end{lem}

\begin{proof}[Proof of Lemma 
\ref{a-lem5.5}]
Let $f\colon (Y, B_Y)\to X$ be a quasi-log 
resolution. 
We replace $S$ with any relatively compact open subset of $S$ 
and take a suitable projective 
bimeromorphic modification of $M$, where $M$ is the ambient 
space of $(Y, B_Y)$. 
Then 
we may assume that the union of  
all strata of $(Y, B_Y)$ mapped to $X^{\star}$, which is 
denoted by $Y''$, is a union of 
some irreducible components of $Y$ (see \cite[Proposition 6.3.1]
{fujino-foundations} and \cite{bierstone-milman2}). 
We put $Y':=Y-Y''$ and $K_{Y'}+B_{Y'}:=(K_Y+B_Y)|_{Y'}$. 
We may further assume that $(Y', f^*E+\Supp B_{Y'})$ is 
an analytic globally embedded simple normal crossing pair. 
We consider 
\begin{equation*}
f\colon \left(Y', B_{Y'}+\varepsilon f^*E\right)\to X
\end{equation*} 
with 
$0<\varepsilon \ll 1$. 
We put $A:=\lceil -(B_Y^{<1})\rceil$ and $N:=\lfloor B_{Y}^{>1}\rfloor$. 
Then $X^{\star}$ is defined by the ideal 
sheaf $f_*\mathcal O_{Y'}(A-N-Y'')$ by the proof of adjunction 
(see Theorem \ref{a-thm4.4}). 
We note that 
\begin{equation*}
\begin{split}
(A-N-Y'')|_{Y'}
&=-\lfloor B_{Y'}+\varepsilon f^*E\rfloor+(B_{Y'}+\varepsilon f^*E)^{=1}
\\&=\lceil -(B_{Y'}+\varepsilon f^*E)^{<1}\rceil 
-\lfloor (B_{Y'}+\varepsilon f^*E)^{>1}\rfloor. 
\end{split}
\end{equation*}
Therefore, if we define 
$\Nqlc (X, \widetilde \omega)$ by the ideal 
sheaf 
\begin{equation*}
f_*\mathcal O_{Y'}(\lceil -(B_{Y'}+\varepsilon f^*E)^{<1}\rceil 
-\lfloor (B_{Y'}+\varepsilon f^*E)^{>1}\rfloor)=f_*\mathcal O_{Y'}
(A-N-Y''), 
\end{equation*} 
then $f\colon (Y', B_{Y'}+\varepsilon f^*E)\to X$ 
gives the desired quasi-log structure on $[X, \widetilde \omega]$. 
\end{proof}

By Lemma \ref{a-lem5.1}, we can prove: 

\begin{lem}\label{a-lem5.6}
Let $\varphi\colon X\to Z$ be a projective surjective 
morphism between complex analytic spaces such that 
$[X, \omega]$ is a quasi-log complex analytic space and 
that $X$ is irreducible. 
Let $P$ be an arbitrary point of $Z$. 
Let $E$ be any positive-dimensional  
irreducible component of $\varphi^{-1}(P)$ with 
$E\not \subset \Nqlc(X, \omega)$. 
Then, after shrinking $Z$ around $P$ suitably, 
we can take an effective $\mathbb R$-Cartier 
divisor $G$ on $Z$ such that 
$[X, \omega+\varphi^*G]$ naturally becomes 
a quasi-log complex analytic space and 
$E$ is a qlc stratum of $[X, \omega+\varphi^*G]$. 
\end{lem}

We will use Lemma \ref{a-lem5.6} when we prove the existence 
of $\omega$-negative 
extremal rational curves (see Theorem \ref{a-thm9.6}). 

\begin{proof}[Sketch of Proof of Lemma \ref{a-lem5.6}]
If $E$ is a qlc stratum of $[X, \omega]$, then 
it is sufficient to put $G=0$. 
From now on, we assume that $E$ is not a qlc stratum of 
$[X, \omega]$. 
We shrink $Z$ around $P$ suitably and take general effective 
Cartier divisors $D_1,\ldots, D_{n+1}$ with 
$P\in \Supp D_i$ for every $i$, where $n=\dim X$. 
By Lemma \ref{a-lem5.2}, 
$\left[X, \omega+\sum _{i=1}^{n+1}\varphi^*D_i\right]$ naturally 
becomes a quasi-log complex analytic space. 
We take a general point $Q\in E$. 
By Lemma \ref{a-lem5.1}, 
we see that $\left[X, \omega+\sum _{i=1}^{n+1}\varphi^*D_i\right]$ 
is not quasi-log canonical at $Q$. 
Therefore, we can find $0<c<1$ such that 
$G:=c\sum _{i=1}^{n+1} D_i$ and $E$ is a qlc center 
of $[X, \omega+\varphi^*G]$. 
This is what we wanted. 
\end{proof}

Similarly, we also have: 

\begin{lem}\label{a-lem5.7}
Let $\varphi\colon X\to Z$ be a projective surjective 
morphism between complex analytic spaces with 
$\dim Z>0$ such that 
$[X, \omega]$ is a quasi-log complex analytic space, 
$X$ is irreducible, and $\Nqlc(X, \omega)=\emptyset$. 
Let $P$ be an arbitrary point of $Z$ with $\dim \varphi^{-1}(P)>0$. 
Then, after shrinking $Z$ around $P$ suitably, 
there exists an effective $\mathbb R$-Cartier 
divisor $G'$ on $Z$ such that 
$[X, \omega+\varphi^*G']$ naturally becomes 
a quasi-log complex analytic space, 
there exists a positive-dimensional 
qlc center $C$ of $[X, \omega+\varphi^*G']$ with 
$\varphi(C)=P$, 
$\dim \Nqlc(X, \omega+\varphi^*G')\leq 0$, and 
$\Nqlc(X, \omega+\varphi^*G')=\emptyset$ outside $\varphi^{-1}(P)$. 
\end{lem}

We will use Lemma \ref{a-lem5.7} in the proof of 
Theorem \ref{a-thm9.4}. 

\begin{proof}[Sketch of Proof of Lemma \ref{a-lem5.7}]
If there exists a positive-dimensional qlc center $C$ of 
$[X, \omega]$ with $\varphi(C)=P$, then it is sufficient to 
put $G'=0$. So we may assume that there are no positive-dimensional 
qlc centers in $\varphi^{-1}(P)$. From now on, 
we will use the same notation as in the proof of Lemma 
\ref{a-lem5.6}. It is not difficult to see that 
we can take $0<c'<1$ such that 
$[X, \omega+\varphi^*G']$, where 
$G':=c'\sum _{i=1}^{n+1} D_i$, 
satisfies all the desired properties. 
\end{proof}

We close this section with the following easy lemma. 

\begin{lem}\label{a-lem5.8}
Let $\left(X, \omega, f\colon 
(Y, B_Y)\to X\right)$ be a quasi-log complex analytic space. 
We assume that the support of $B_Y$ has only finitely many 
irreducible components. 
Then we obtain a $\mathbb Q$-divisor $D_i$ on $Y$, 
a $\mathbb Q$-line bundle $\omega_i$ on $X$, and a positive 
real number $r_i$ for $1\leq i\leq k$ such that 
\begin{itemize}
\item[(i)] $\sum _{i=1}^k r_i=1$, 
\item[(ii)] $\Supp D_i=\Supp B_Y$, 
$D^{=1}_i=B^{=1}_Y$, $\lfloor D^{>1}_i\rfloor=
\lfloor B^{>1}_Y\rfloor$, 
and $\lceil -(D^{<1}_i)\rceil=\lceil -(B^{<1}_Y)\rceil$ 
for every $i$, 
\item[(iii)] $\omega=\sum _{i=1}^k r_i\omega_i$ holds 
in $\Pic(X)\otimes _{\mathbb Z}\mathbb R$ 
and 
$B_Y=\sum _{i=1}^k r_i D_i$, and 
\item[(iv)] $\left(X, \omega_i, f\colon (Y, D_i)\to X\right)$ is a quasi-log 
complex analytic space 
with $K_Y+D_i\sim _{\mathbb Q} f^*\omega_i$ for every $i$. 
\end{itemize}
We note that 
\begin{equation*}
\mathcal I_{\Nqlc(X, \omega_i)}=\mathcal I_{\Nqlc(X, \omega)}
\end{equation*} 
holds for every $i$. 
In particular, 
if $\Nqlc(X, \omega)=\emptyset$, then 
$\Nqlc(X, \omega_i)=\emptyset$ for every $i$. 
We also note that $W$ is a qlc stratum of 
$[X, \omega]$ if and only if 
$W$ is a qlc stratum of $[X, \omega_i]$ for 
every $i$. 
\end{lem}

\begin{proof}[Proof of Lemma \ref{a-lem5.8}]
The proof of \cite[Lemma 4.25]{fujino-cone} works 
without any changes. 
\end{proof}

\section{Basepoint-free theorem}\label{a-sec6}
In this section, we will prove the 
following basepoint-free theorem 
for quasi-log complex analytic spaces, 
which is a generalization of 
\cite[Theorem 6.5.1]{fujino-foundations} and 
\cite[Theorem 4.2.1]{fujino-cone-contraction}. 

\begin{thm}[Basepoint-free theorem for 
quasi-log complex analytic spaces]\label{a-thm6.1}
Let 
\begin{equation*}
\left(X, \omega, f\colon (Y, B_Y)\to X\right)
\end{equation*} 
be a quasi-log complex 
analytic space and let $\pi\colon X\to S$ be 
a projective morphism between complex analytic spaces. 
Let $W$ be a compact subset of $S$ and 
let $\mathcal L$ be a line bundle 
on $X$ such that $\mathcal L$ is $\pi$-nef over $W$. 
We assume that 
\begin{itemize}
\item[(i)] $q\mathcal L-\omega$ is $\pi$-ample over 
$W$ for some real number 
$q>0$, and 
\item[(ii)] there exists some open neighborhood 
of $W$ over which $\mathcal L^{\otimes m}|_{X_{-\infty}}$ is 
$\pi|_{X_{-\infty}}$-generated for every $m\gg 0$. 
\end{itemize} 
Then there 
exists a relatively 
compact open neighborhood $U$ of 
$W$ such that 
$\mathcal L^{\otimes m}$ is $\pi$-generated over $U$ for every $m\gg0$. 
\end{thm}

\begin{proof}
By shrinking $S$ around $W$ suitably,
 we may assume that $\mathcal L^{\otimes m}|_{X_{-\infty}}$ is 
 $\pi|_{X_{-\infty}}$-generated for every $m\gg 0$ by (ii). 
 We take an arbitrary point $w\in W$. 
 Then it is sufficient to prove that $\mathcal L^{\otimes m}$ 
 is $\pi$-generated for every $m\gg 0$ over some relatively 
 compact open neighborhood of $w$ since 
 $W$ is compact. 
 Hence, we may assume that $W=\{w\}$, 
 $S$ is Stein, and $\pi$ is surjective. 
 We will sometimes shrink $S$ around $W$ suitably without mentioning 
 it explicitly throughout 
 this proof. 
 We use induction on the dimension of $X\setminus X_{-\infty}$. 
 We note that Theorem \ref{a-thm6.1} obviously holds true when 
 $\dim X\setminus X_{-\infty}=0$. 
\setcounter{step}{0}
\begin{step}\label{a-6.1-step1}
In this step, we will prove that for every $m\gg 0$ there exists 
an open neighborhood $U_m$ of $w$ such that 
$\mathcal L^{\otimes m}$ is $\pi$-generated around $\Nqklt(X, \omega)$ 
over $U_m$. 

We put $X':=\Nqklt(X, \omega)$. 
Then, 
after shrinking $S$ around $w$ suitably, 
$[X', \omega']$, where $\omega':=\omega|_{X'}$, is 
a quasi-log complex analytic space by adjunction 
when $X'\ne X_{-\infty}$ (see Theorem \ref{a-thm4.4}). 
If $X'= X_{-\infty}$, then $\mathcal L^{\otimes m}|_{X'}$ is 
$\pi$-generated for every $m \gg 0$ by assumption. 
If $X'\ne X_{-\infty}$, then $\mathcal L^{\otimes m}|_{X'}$ is 
$\pi$-generated for every $m\gg 0$ by induction on 
the dimension of $X\setminus X_{-\infty}$ after shrinking $S$ around 
$w$ suitably. We can take an open neighborhood $U_m$ of $w$ such that 
$m\mathcal L-\omega$ is $\pi$-ample over $U_m$. 
Then $R^1\pi_*\left(\mathcal I_{X'}\otimes \mathcal L^{\otimes m}\right)=0$ 
on $U_m$ 
(see Theorems \ref{a-thm4.7} and 
\ref{a-thm4.8}). Thus, the restriction map 
\begin{equation*}
\pi_*\mathcal L^{\otimes m}\to \pi_*\left(\mathcal L^{\otimes m}|_{X'}\right)
\end{equation*} 
is surjective on $U_m$. 
This implies that $\mathcal L^{\otimes m}$ is $\pi$-generated 
around $\Nqklt(X, \omega)$ over $U_m$. 
This is what we wanted. 
\end{step}
\begin{step}\label{a-6.1-step2}
In this step, we will prove that 
$\pi_*\left(\mathcal L^{\otimes m}|_{X'}\right)
\ne 0$ for 
every $m\gg 0$ when $X'\cap \Nqklt(X, \omega)$ is empty, 
where $X'$ is any connected component of $X$ with $w\in \pi(X')$. 

Throughout this step, we can freely shrink $S$ around $w$. 
Without loss of generality, we may assume that $X$ itself is connected 
(see \cite[Lemma 6.3.12]{fujino-foundations}). 
Then, by Lemma \ref{a-lem4.9}, $X$ is a normal complex variety. 
We apply \cite[Lemma 4.1.3]{fujino-cone-contraction} to 
\begin{equation*}
\left(X, \omega, f\colon (Y, B_Y)\to X\right). 
\end{equation*} 
Let $s$ be an analytically sufficiently general point of $S$. 
Then 
\begin{equation*}
\left(X_s, \omega|_{X_s}, f_s\colon (Y_s, B_{Y_s})\to X_s\right)
\end{equation*}
is a projective quasi-log canonical pair, 
where $X_s:=\pi^{-1}(s)$, $Y_s:=(\pi\circ f)^{-1}(s)$, 
$f_s:=f|_{X_s}$, and $B_{Y_s}:=B_Y|_{Y_s}$. 
We may assume that $\mathcal L|_{X_s}$ is nef 
(see \cite[Lemma 2.2.5 and Remark 2.2.7]{fujino-cone-contraction}) and 
$q\mathcal L|_{X_s}-\omega|_{X_s}$ is ample. 
By the basepoint-free theorem for quasi-log schemes 
(see \cite[Theorem 6.5.1]{fujino-foundations}), 
we obtain that $\mathcal L^{\otimes m}|_{X_s}$ is basepoint-free 
for every $m\gg 0$. 
In particular, $|\mathcal L^{\otimes m}|_{X_s}|\ne \emptyset$ 
for every $m\gg 0$. 
This implies that $\pi_*\mathcal L^{\otimes m}\ne \emptyset$ 
for every $m\gg 0$. This is what we wanted. 
\end{step}
\begin{step}\label{a-6.1-step3}
Let $p$ be a prime number and let $k$ be a large positive 
integer. 
By Steps \ref{a-6.1-step1} and \ref{a-6.1-step2}, 
after shrinking $S$ around $w$ suitably, 
we obtain that $\pi_*\mathcal L^{\otimes p^k}\ne 0$ and that 
$\mathcal L^{\otimes p^k}$ is $\pi$-generated 
around $\Nqklt(X, \omega)$. In this step, we will prove that 
if the relative base locus $\Bs_\pi |\mathcal L^{\otimes p^k}|$ 
with the reduced structure is not empty over $w$ 
then, after shrinking $S$ around $w$ suitably again, there exists a positive 
integer $l>k$ such that $\Bs_\pi |\mathcal L^{\otimes p^l}|$ is strictly smaller 
than $\Bs_\pi|\mathcal L^{\otimes p^k}|$.  

From now on, we will sometimes shrink $S$ around $w$ suitably without 
mentioning it explicitly. 
Let $f\colon (Y, B_Y)\to X$ be a quasi-log resolution. 
We take a general member $D\in |\mathcal L^{\otimes p^k}|$. 
Then we may assume that $f^*D$ intersects any strata of $(Y, \Supp B_Y)$ 
transversally over $X\setminus \Bs_\pi|\mathcal L^{\otimes p^k}|$ 
by Bertini's theorem and that $f^*D$ contains no strata of $(Y, B_Y)$. 
By taking a suitable projective bimeromorphic modification 
of $M$, 
the ambient space of $(Y, B_Y)$, 
we may assume that $(Y, f^*D+\Supp B_Y)$ is an analytic 
globally embedded simple normal crossing pair 
(see \cite{bierstone-milman2} and 
\cite[Proposition 6.3.1]{fujino-foundations}). 
After shrinking $S$ around $w$ suitably,
 we take the maximal positive real number $c$ such that 
 $B_Y+cf^*D$ is a subboundary over $X\setminus X_{-\infty}$. 
 We note that $c\leq 1$ holds. 
 Here, we used the fact that the natural map $\mathcal O_X\to 
 f_*\mathcal O_Y(\lceil -(B_Y^{<1})\rceil)$ is an isomorphism 
 over $X\setminus X_{-\infty}$ (see \cite[Claim 3.5]{fujino-relative-span}). 
 Then 
 \begin{equation*}
 f\colon (Y, B_Y+cf^*D)\to X
 \end{equation*} 
 gives a natural quasi-log 
 structure on the pair $[X, \omega':=\omega+cD]$ 
 (see Lemma \ref{a-lem5.2}, \cite[Proposition 3.4]{fujino-relative-span}, 
 and so on). 
 We note that $\Nqlc(X, \omega)=\Nqlc(X, \omega')$ holds by construction. 
 We note that we may assume that 
 $[X, \omega']$ has a qlc center $C$ that intersects 
 $\Bs_\pi|\mathcal L^{\otimes p^k}|\cap \pi^{-1}(w)$. 
 Since $(q+cp^k)\mathcal L-\omega'$ is $\pi$-ample 
 over $W$, for every $m\gg 0$, there exists some open neighborhood 
 $U'_m$ of $w$ such that 
 $\mathcal L^{\otimes m}$ is $\pi$-generated around $\Nqklt(X, \omega')$ 
 over $U'_m$ by Step \ref{a-6.1-step1}. 
 In particular, $\mathcal L^{\otimes m}$ is $\pi$-generated 
 around $C$ over $U'_m$. 
 Thus, $\Bs_\pi|\mathcal L^{\otimes p^l}|$ is strictly smaller than 
 $\Bs_\pi|\mathcal L^{\otimes p^k}|$ for some positive integer $l>k$. 
This is what we wanted. 
\end{step}

\begin{step}\label{a-6.1-step4}
In this step, we will complete the proof. 
 
By using Step \ref{a-6.1-step3} finitely many times, 
we obtain an open neighborhood $V_p$ of $w$ and a 
large positive integer $n$ such that 
$\mathcal L^{\otimes p^n}$ is $\pi$-generated over 
$V_p$. 
We take another prime number $p'$. 
By the same argument, 
we obtain an open neighborhood 
$V_{p'}$ of $w$ and a large positive integer $n'$ such that 
$\mathcal L^{\otimes p'^{n'}}$ is $\pi$-generated 
over $V_{p'}$. 
Hence, we can take an open neighborhood 
$U_w$ of $w$ and a positive integer $m_0$ such that 
$\mathcal L^{\otimes m}$ is $\pi$-generated 
over $U_w$ for every $m\geq m_0$ (see 
Lemma \ref{a-lem8.6} below). 
\end{step}
As we mentioned above, we can obtain a desired open neighborhood 
$U$ of $W$ since $W$ is compact. 
We finish the proof. 
\end{proof}

By combining Theorem \ref{a-thm6.1} with 
Example \ref{a-ex1.5}, we can recover the basepoint-free 
theorem for normal pairs (see 
\cite[Theorem 4.2.1]{fujino-cone-contraction}). 

\begin{rem}\label{a-rem6.2}
For normal pairs, we first established the non-vanishing 
theorem (see \cite[Theorem 4.1.1]{fujino-cone-contraction}) 
and then proved the basepoint-free theorem 
(see \cite[Theorem 4.2.1]{fujino-cone-contraction}). 
On the other hand, in this section, we can directly prove 
the basepoint-free theorem (see Theorem \ref{a-thm6.1}) 
because 
quasi-log structures behave well for inductive treatments. 
\end{rem}

\section{Basepoint-free theorem of Reid--Fukuda type}\label{a-sec7} 

In this section, we will prove the basepoint-free 
theorem of Reid--Fukuda type. 
If we apply this theorem to kawamata log terminal pairs, 
then we can recover the Kawamata--Shokurov basepoint-free 
theorem for projective morphisms between complex analytic 
spaces. 

\begin{thm}[Basepoint-free theorem of Reid--Fukuda 
type for quasi-log complex analytic spaces]\label{a-thm7.1}
Let $[X, \omega]$ be a quasi-log complex analytic 
space and let $\pi\colon X\to S$ be a projective morphism 
of complex analytic spaces. 
Let $\mathcal L$ be a $\pi$-nef line bundle 
on $X$ such that $q\mathcal L-\omega$ is 
nef and 
log big over $S$ with respect to $[X, \omega]$ for 
some positive real number $q$. 
This means that $q\mathcal L-\omega$ is nef over $S$ and 
that $(q\mathcal L-\omega)|_C$ is big over $\pi(C)$ 
for every qlc stratum $C$ of $[X, \omega]$. 
We assume that $\mathcal L^{\otimes m}|_{X_{-\infty}}$ 
is $\pi$-generated for every $m \gg 0$. 
Then, after replacing $S$ with 
any relatively compact open subset of 
$S$, $\mathcal L^{\otimes m}$ is $\pi$-generated 
for every $m \gg 0$. 
\end{thm}

The following proof is essentially the same as 
the one for \cite[Theorem 1.1]{fujino-reid-fukuda}. 

\begin{proof}[Proof of Theorem \ref{a-thm7.1}]
In Step \ref{a-7.1-step1}, we will reduce the problem to 
the case where $X\setminus X_{-\infty}$ is irreducible 
and the relative base locus 
of $\mathcal L^{\otimes m}$ is 
disjoint from $\Nqklt(X, \omega)$ 
for every $m \gg 0$. 
\setcounter{step}{0}
\begin{step}\label{a-7.1-step1}
We use induction on $\dim (X\setminus X_{-\infty})$. 
It is obvious that the statement holds when 
$\dim (X\setminus X_{-\infty})
=0$. 
We take an arbitrary point $P\in S$. 
It is sufficient to prove the statement over some 
open neighborhood of $P$. Hence 
we will freely shrink $S$ around $P$ throughout this proof. 
In particular, we may assume that $S$ is Stein. 
Let $C$ be any qlc stratum of $[X, \omega]$. 
We put $X':=C\cup \Nqlc(X, \omega)$. 
By adjunction (see Theorem \ref{a-thm4.4}), after replacing $S$ with 
any relatively compact open neighborhood of 
$P$, $[X', \omega':=\omega|_{X'}]$ is 
a quasi-log complex analytic space. 
By the vanishing theorem 
(see Theorem \ref{a-thm4.8}), we have $R^1\pi_*\left(\mathcal I_{X'}
\otimes \mathcal L^{\otimes m}\right)=0$ 
for every $m\geq q$. 
Therefore, the natural restriction map 
\begin{equation*}
\pi_*\mathcal L^{\otimes m}\to 
\pi_*\left(\mathcal L^{\otimes m}|_{X'}\right)
\end{equation*} 
is 
surjective for every $m\geq q$. 
This implies that we may assume that $X\setminus X_{-\infty}$ is 
irreducible by replacing $X$ with $X'$. 
If $\Nqklt(X, \omega)=\Nqlc(X, \omega)$, 
then $\mathcal L^{\otimes m}|_{\Nqklt(X, \omega)}$ is 
$\pi$-generated for every $m\gg 0$ by assumption. 
If $\Nqklt(X, \omega)\ne \Nqlc(X, \omega)$, 
then we know that 
$\mathcal L^{\otimes m}|_{\Nqklt(X, \omega)}$ is 
$\pi$-generated for every $m\gg 0$ 
by induction on $\dim (X\setminus X_{-\infty})$. 
By the vanishing theorem again 
(see Theorems \ref{a-thm4.7} and \ref{a-thm4.8}), 
we have 
$R^1\pi_*\left(\mathcal I_{\Nqklt(X, \omega)}
\otimes \mathcal L^{\otimes m}\right)=0$ for 
every $m\geq q$. 
Thus, the restriction map 
\begin{equation*}
\pi_*\mathcal L^{\otimes m}\to 
\pi_*\left(\mathcal L^{\otimes m}|_{\Nqklt(X, \omega)}\right)
\end{equation*} 
is 
surjective for every $m\geq q$. 
Hence, the relative base locus $\Bs_{\pi}|\mathcal L^{\otimes m}|$ 
of $\mathcal L^{\otimes m}$ 
is disjoint from $\Nqklt(X, \omega)$ for every $m\gg 0$. 
\end{step}
\begin{step}\label{a-7.1-step2}
In this step, we will prove the basepoint-freeness 
under the extra assumption that 
$X$ is the disjoint union of $X_{-\infty}=\Nqlc(X, \omega)$ and 
a qlc stratum $C$ of $[X, \omega]$ such that 
$C$ is the unique qlc stratum of $[X, \omega]$. 

In the above setting, we may assume that $X_{-\infty}=\emptyset$ 
(see \cite[Lemma 6.3.12]{fujino-foundations}). 
By Kodaira's lemma, after replacing $S$ with 
any relatively compact open neighborhood of 
$P$, we can write 
\begin{equation*}
q\mathcal L-\omega\sim _{\mathbb R} A+E
\end{equation*} 
on $X$ such that 
$A$ is a $\pi$-ample $\mathbb Q$-divisor on $X$ and 
$E$ is an effective $\mathbb R$-Cartier divisor on $X$. 
We put $\widetilde \omega=\omega+\varepsilon E$ with 
$0<\varepsilon \ll 1$. 
Then $[X, \widetilde\omega]$ is a quasi-log complex analytic space
with $\Nqlc(X, \widetilde \omega)=
\emptyset$ by Lemma \ref{a-lem5.5}. 
We note that 
\begin{equation*}
q\mathcal L-\widetilde \omega
\sim _{\mathbb R}(1-\varepsilon)(q\mathcal L-\omega)+\varepsilon A
\end{equation*}
is $\pi$-ample. 
Therefore, by the basepoint-free theorem (see Theorem 
\ref{a-thm6.1}), 
we obtain that 
$\mathcal L^{\otimes m}$ is $\pi$-generated for every $m\gg 0$ 
over some open neighborhood of $P$.
\end{step}
\begin{step}\label{a-7.1-step3}
By Step \ref{a-7.1-step2}, we may assume that 
$X$ is connected and $\Nqklt(X, \omega)\ne \emptyset$. 
Let $p$ be any prime number. Then, by Step \ref{a-7.1-step1}, 
the relative base locus $\Bs_{\pi}|p^l\mathcal L|$ 
of $p^l\mathcal L$ is strictly smaller than $X$ for some large 
positive integer $l$. 
In this step, we will prove the following claim. 
\begin{claim}
If the relative base locus $\Bs_{\pi}|p^l\mathcal L|$ 
with the reduced structure is not 
empty over $P$, then there is a positive integer $k$ with 
$k>l$ such that 
$\Bs_{\pi}|p^k\mathcal L|$ is strictly smaller than 
$\Bs_{\pi}|p^l\mathcal L|$ after shrinking $S$ around $P$ suitably. 
\end{claim}
\begin{proof}[Proof of Claim] 
Note that the inclusion $\Bs_{\pi}|p^k \mathcal 
L|\subseteq \Bs_{\pi}|p^l\mathcal L|$ obviously holds 
for every 
positive integer $k>l$. 
Let us consider $[X^\dag, \omega^\dag]$ 
as in Lemma \ref{a-lem5.3}.
Since $(q\mathcal L-\omega)|_{X^\dag}$ is nef 
and big over $S$, after replacing $S$ with 
any relatively compact open neighborhood of $P$, we can 
write 
\begin{equation*}
q\mathcal L|_{X^\dag}-\omega^\dag
\sim_{\mathbb R}A+E
\end{equation*} 
on $X^\dag$ by Kodaira's lemma, 
where $A$ is a $\pi$-ample 
$\mathbb Q$-divisor on $X^\dag$ and 
$E$ is an effective $\mathbb R$-Cartier 
divisor 
on $X^\dag$. 
By Lemma \ref{a-lem5.5}, after 
replacing $S$ with any relatively compact Stein open neighborhood of 
$P$, we have a new 
quasi-log structure on $[X^\dag, \widetilde \omega]$, 
where $\widetilde \omega=\omega^\dag+\varepsilon E$ 
with $0<\varepsilon \ll 1$, such that 
\begin{equation}\label{a-eq7.1}
\Nqlc(X^\dag, \widetilde {\omega})=\left(\bigcup_{i\in I} C_i\right)\cup 
\Nqlc(X^\dag, \omega^\dag), 
\end{equation} 
where $\{C_i\}_{i\in I}$ is the set of qlc centers of $[X^\dag, 
\omega^\dag]$ contained 
in $\Supp E$. We put $n:=\dim X^\dag$. Let 
$D_1, \ldots, D_{n+1}$ be general members of 
$|p^l\mathcal L|$. 
Let $f\colon (Y, B_Y)\to X^\dag$ be a quasi-log 
resolution of $[X^\dag, \widetilde
\omega]$. 
We consider 
\begin{equation*}
f\colon \left(Y, B_Y+\sum _{i=1}^{n+1}f^*D_i\right)\to X^\dag. 
\end{equation*} 
Without loss of generality, we may assume that 
\begin{equation*}
\left(Y, \sum _{i=1}^{n+1}f^*D_i+\Supp B_Y\right)
\end{equation*} 
is an analytic globally embedded simple normal crossing pair 
by taking a suitable projective bimeromorphic 
modification of the ambient space of $(Y, B_Y)$ 
(see \cite{bierstone-milman2} and 
\cite[Proposition 6.3.1]{fujino-foundations}). 
Then, after shrinking $S$ around $P$ suitably, we 
can take $0<c<1$ 
such that 
\begin{equation*}
f\colon \left(Y, B_Y+c\sum _{i=1}^{n+1}f^*D_i\right)
\to X^\dag
\end{equation*} 
gives a quasi-log 
structure on $\left[X^\dag, \widetilde {\omega}+
c\sum _{i=1}^{n+1}D_i\right]$ 
such that 
$\left[X^\dag, \widetilde {\omega}+c\sum _{i=1}^{n+1}D_i\right]$ 
has only quasi-log canonical singularities on $X^\dag\setminus 
\Nqlc(X^\dag, \widetilde \omega)$ and that 
there exists a qlc center $C_0$ of $\left[X^\dag, 
\widetilde \omega+c
\sum _{i=1}^{n+1}D_i\right]$ contained in $\Bs_{\pi}|p^l\mathcal 
L|$ 
with $C_0\cap \pi^{-1}(P)\ne \emptyset$ 
(see Lemmas \ref{a-lem5.1} and \ref{a-lem5.2}). 
We put $\widetilde \omega+c
\sum _{i=1}^{n+1}D_i=\overline \omega$. 
Then, by construction, 
\begin{equation*}
C_0\cap \Nqlc(X^\dag, \overline \omega)=\emptyset 
\end{equation*} 
holds 
because 
\begin{equation*}
\Bs_{\pi}|p^l\mathcal L|\cap \Nqklt (X, \omega)=\emptyset. 
\end{equation*}
Note that $\Nqlc(X^\dag, \overline \omega)=\Nqlc(X^\dag, \widetilde \omega)$ 
by construction. 
We also note that 
\begin{equation*}
(q+c(n+1)p^l)\mathcal 
L|_{X^\dag}-\overline \omega\sim _{\mathbb R}(1-\varepsilon)
(q\mathcal L|_{X^\dag}-\omega^\dag)+\varepsilon A
\end{equation*} 
is ample over $S$. 
Therefore, 
\begin{equation}\label{a-eq7.2}
\pi_*\left(\mathcal L^{\otimes m}|_{X^\dag}\right)\to \pi_*\left(
\mathcal L^{\otimes m}|_{C_0}\right)
\oplus \pi_*\left(\mathcal L^{\otimes m}|_{\Nqlc(X^\dag, \overline 
\omega)}\right)
\end{equation}
is surjective for every $m\geq q+c(n+1)p^l$ 
by Theorem \ref{a-thm4.8}. 
Moreover, $\mathcal L^{\otimes m}|_{C_0}$ is $\pi$-generated 
for every $m\gg 0$ by the basepoint-free theorem 
(see Theorem \ref{a-thm6.1}). 
Note that $\left[C_0, {\overline \omega}|_{C_0}\right]$ is 
a quasi-log complex analytic space with only quasi-log canonical singularities 
(see \cite[Lemma 6.3.12]{fujino-foundations}). 
Therefore, we can construct a section $s$ 
of $\mathcal L^{\otimes p^k}|_{X^\dag}$ 
for some positive integer $k>l$ such that 
$s|_{C_0}$ is not zero and $s$ is 
zero on $\Nqlc(X^\dag, \overline{\omega})$ 
by \eqref{a-eq7.2}. 
Thus $s$ is zero on 
\begin{equation*} 
\Nqlc(X^\dag, \overline{\omega})=\Nqlc 
(X^\dag, \widetilde{\omega})=\left(\bigcup_{i\in I} C_i\right)
\cup \Nqlc (X^\dag, 
\omega^\dag)
\end{equation*} 
by \eqref{a-eq7.1}. 
In particular, $s$ is zero on $\Nqlc(X^\dag, \omega^\dag)$. 
Hence, $s$ can be seen as a section of $\mathcal L^{\otimes p^k}$ 
because 
$\mathcal I_{\Nqlc(X^\dag, \omega^\dag)}=\mathcal I_{\Nqlc(X, \omega)}$ 
by construction (see Lemma \ref{a-lem5.3}). 
More precisely, 
we can see 
\begin{equation*}
s\in \pi_*\left(\mathcal I_{\Nqlc(X^\dag, \overline {\omega})}
\otimes \mathcal L^{\otimes p^k}\right)
\end{equation*} 
by 
construction. 
Since 
\begin{equation*}
\mathcal I_{\Nqlc(X^\dag, \overline {\omega})}
\subset \mathcal I_{\Nqlc(X^\dag, \omega^\dag)}
=\mathcal I_{\Nqlc(X, \omega)}, 
\end{equation*} 
we have 
\begin{equation*}
s\in \pi_*\left(\mathcal I_{\Nqlc(X^\dag, \overline {\omega})}
\otimes \mathcal L^{\otimes p^k}\right)
\subset 
\pi_*\left(\mathcal I_{\Nqlc(X, \omega)}
\otimes \mathcal L^{\otimes p^k}\right)
\subset \pi_*\left(\mathcal L^{\otimes p^k}\right). 
\end{equation*} 
Therefore, $\Bs_{\pi}|p^k\mathcal L|$ is strictly smaller 
than $\Bs_{\pi}|p^l\mathcal L|$ over $P$. 
We complete the proof of Claim. 
\end{proof}
\end{step}
\begin{step}\label{a-7.1-step4}
By the noetherian induction, after shrinking $S$ 
around $P$ suitably, $p^l\mathcal L$ and 
$p'^{l'}\mathcal L$ are both $\pi$-generated for large 
positive integers 
$l$ and $l'$, where $p$ and $p'$ are distinct prime numbers. 
Hence there exists a positive integer $m_0$ such that 
$\mathcal L^{\otimes m}$ is $\pi$-generated for every $m\geq m_0$ 
(see Lemma \ref{a-lem8.6} below). 
\end{step}
We finish the proof. 
\end{proof}

As we saw above, since $\pi\colon X\to S$ is projective in 
Theorem \ref{a-thm7.1}, the proof of 
\cite[Theorem 1.1]{fujino-reid-fukuda} works even 
when $\pi\colon X\to S$ is not algebraic. When $\pi\colon X\to S$ 
is algebraic but is only {\em{proper}}, the proof 
of Theorem \ref{a-thm7.1} is unexpectedly 
difficult. For the details, see the proof of 
\cite[Theorem 1.1]{fujino-quasi}. 

\section{Effective freeness}\label{a-sec8}

In this section, we will prove the following effective freeness 
and effective very ampleness. 
This type of effective freeness was originally due to Koll\'ar (see 
\cite{kollar}). Note that his method was already generalized for 
quasi-log schemes in \cite{fujino-effective}. 
Here, we give a slightly simpler proof for quasi-log complex 
analytic spaces. 

\begin{thm}[Effective freeness for quasi-log complex 
analytic spaces]\label{a-thm8.1}
Let $[X, \omega]$ be a quasi-log complex analytic space 
with $X_{-\infty}=\emptyset$ and let 
$\pi\colon X\to S$ be a projective morphism 
between complex analytic spaces. 
Let $\mathcal L$ be a $\pi$-nef line bundle on $X$ such that 
$a\mathcal L-\omega$ is $\pi$-ample over $S$ 
for some non-negative integer $a$. Then there 
exists a positive integer $m=m(\dim X, a)$, which 
only depends on $\dim X$ and $a$, such that 
$\mathcal L^{\otimes m}$ is $\pi$-generated. 
Moreover, there exists a positive integer $m_0=m_0(\dim X, a)$ 
depending only 
on $\dim X$ and $a$ such that 
$\mathcal L^{\otimes l}$ is $\pi$-generated for every $l\geq m_0$. 
\end{thm}

If $\omega$ is a $\mathbb Q$-line bundle in Theorem 
\ref{a-thm8.1}, then we have: 

\begin{thm}\label{a-thm8.2} 
In Theorem \ref{a-thm8.1}, 
if $\omega$ is a $\mathbb Q$-line bundle {\em{(}}or a 
globally $\mathbb Q$-Cartier divisor{\em{)}}, then 
we may replace the assumption that 
$a\mathcal L-\omega$ is $\pi$-ample 
over $S$ with a weaker one that $a\mathcal L-\omega$ is nef and 
log big over $S$ with respect to $[X, \omega]$. 
\end{thm}

If $\mathcal L$ is $\pi$-ample in Theorem \ref{a-thm8.1}, 
then we have the following effective very ampleness. 

\begin{thm}[Effective very ampleness for 
quasi-log complex analytic spaces]\label{a-thm8.3} 
Let $[X, \omega]$ be a quasi-log complex analytic space 
with $X_{-\infty}=\emptyset$ and let 
$\pi\colon X\to S$ be a projective morphism 
between complex analytic spaces. 
Let $\mathcal L$ be a $\pi$-ample line bundle on $X$ such that 
$a\mathcal L-\omega$ is $\pi$-nef over $S$ 
for some non-negative integer $a$. 
Then there exists a positive integer $m'=m'(\dim X, a)$ depending only on 
$\dim X$ and $a$ such that 
$\mathcal L^{\otimes m'}$ is $\pi$-very ample. 
Moreover, there exists a positive integer $m'_0=m'_0(\dim X, a)$ depending 
only on $\dim X$ and $a$ such that 
$\mathcal L^{\otimes l'}$ is $\pi$-very ample for every $l'\geq m'_0$. 
\end{thm}

We will use the following easy lemmas in the 
proof of Theorem \ref{a-thm8.1}. 

\begin{lem}\label{a-lem8.4} 
Let $P(x)$ be a polynomial and let $a$ and $n$ be positive 
integers. 
Assume that, with at most $n$ exceptions, 
$P(a+j)\ne 0$ holds for 
every non-negative integer $j$. 
Then, for every positive integer $m\geq 2(a+n)$, 
there exists a non-negative 
integer $j_0$ with $0\leq j_0\leq n$ such that 
$P(a+j_0)\ne 0$ and $P(m-a-j_0)\ne 0$. 
\end{lem}

\begin{proof}
We note that $m-a-j=a+(m-2a)-j$ and $m-2a\geq 2n$. 
Therefore, we can easily find some non-negative integer $j_0$ 
with $0\leq j_0\leq n$ such that 
$P(a+j_0)\ne 0$ and $P(m-a-j_0)\ne 0$. 
\end{proof}

\begin{lem}\label{a-lem8.5} 
Let $n_0$ and $n_1$ be positive integers such that $\gcd(n_0, n_1)=1$. 
We put 
$n_2:=(kn_0+1)n_1$, where $k$ is any positive integer. 
Then $\gcd(n_0, n_2)=1$. 
\end{lem}

\begin{proof}
It is obvious. 
\end{proof}

\begin{lem}\label{a-lem8.6}
Let $a$ and $b$ be positive integers with 
$1<a<b$ such that $\gcd(a, b)=1$. 
Then, for any positive integer $l$ with 
$l\geq a\left(b-\lceil \frac{b}{a}\rceil\right)$, 
there exist non-negative integers $u$ 
and $v$ such that $l=ua+vb$. 
\end{lem}

\begin{proof} 
It is an easy exercise. 
For the details, see, for example, the proof of 
\cite[Lemma 4.3]{fujino-quasi}. 
\end{proof}

Let us prove Theorem \ref{a-thm8.1}. 
The proof is new and is slightly simpler than the one given in 
\cite{fujino-effective} for quasi-log schemes. 

\begin{proof}[Proof of Theorem \ref{a-thm8.1}]
We take an arbitrary point $s\in S$. 
It is sufficient to prove the existence of $m(\dim X, a)$ and 
$m_0(\dim X, a)$ 
over some open neighborhood of $s$. 
Therefore, we replace $S$ with a relatively compact Stein open 
neighborhood of $s$. 
\setcounter{step}{0}
\begin{step}\label{a-8.1-step1}
Let $X'$ be an irreducible component of $X$. 
Then $X'$ is a qlc stratum of $[X, \omega]$. 
Hence, by adjunction (see Theorem \ref{a-thm4.4}), we see that 
$[X', \omega':=\omega|_{X'}]$ is a 
quasi-log complex analytic space with $X'_{-\infty}=\emptyset$. 
By the vanishing theorem (see Theorem \ref{a-thm4.8}), we have 
\begin{equation*}
R^1\pi_*(\mathcal I_{X'}\otimes 
\mathcal L^{\otimes j})=0
\end{equation*} 
for every $j\geq a$. 
Thus the natural restriction map 
\begin{equation*}
\pi_*\mathcal L^{\otimes j}\to 
\pi_*(\mathcal L^{\otimes j}|_{X'})
\end{equation*} 
is surjective for 
every $j\geq a$. 
Therefore, we replace $X$ with $X'$ and may assume 
that $X$ is irreducible. 
\end{step}
\begin{step}\label{a-8.1-step2}
In this step, we will prove the following claims. 
\setcounter{cla}{0}
\begin{cla}[{see \cite[Lemma 3.3]{fujino-effective}}]\label{a-8.1-claim1}
For every positive integer $m_1\geq 2(a+\dim X)$, 
there exists an effective Cartier divisor $D_1$ on $X$ such that 
$D_1\in |\mathcal L^{\otimes m_1}|$ and that $\Supp D_1$ contains 
no qlc 
strata of $[X, \omega]$. 
\end{cla}
\begin{proof}[Proof of Claim \ref{a-8.1-claim1}]
Let $C$ be any qlc stratum of $[X, \omega]$. We consider the following 
short exact sequence: 
\begin{equation*}
0\to \mathcal I_C\otimes \mathcal L^{\otimes j} 
\to \mathcal L^{\otimes j} \to \mathcal L^{\otimes j}|_C\to 0,  
\end{equation*} 
where $\mathcal I_C$ is the defining ideal sheaf of $C$ on $X$. 
By the vanishing theorem (see Theorem \ref{a-thm4.8}), 
\begin{equation*}
R^i\pi_*\left(\mathcal I_C\otimes \mathcal L^{\otimes j}\right)
=R^i\pi_*\mathcal L^{\otimes j} =R^i\pi_*\left(\mathcal L^{\otimes j}|_C\right)=0
\end{equation*} 
for every $i\geq 1$ and $j\geq a$. 
Therefore, 
\begin{equation*}
\chi (C_s, \mathcal L^{\otimes j}|_{C_s})=
\dim H^0(C_s, \mathcal L^{\otimes j}|_{C_s})
\end{equation*} 
holds for $j\geq a$, where $C_s$ is an analytically 
sufficiently general fiber of $C\to \pi(C)$. 
Note that $\chi (C_s, \mathcal L^{\otimes j}|_{C_s})$ is 
a non-zero polynomial in $j$ since $\mathcal L^{\otimes m}$ is $\pi$-generated 
for every $m\gg 0$ by the basepoint-free theorem 
(see Theorem \ref{a-thm6.1}). 
We also note that the restriction map 
\begin{equation}\label{a-eq8.1}
\pi_*\mathcal L^{\otimes j}\to \pi_*\left(\mathcal L^{\otimes j}|_C\right)
\end{equation} 
is surjective for every $j\geq a$. 
Thus, with at most $\dim C_s$ exceptions, 
we have 
\begin{equation*}
\dim H^0(C_s, \mathcal L^{\otimes (a+j)}|_{C_s})\ne 
0
\end{equation*} for $j\geq 0$.   
By Lemma \ref{a-lem8.4}, we see that 
\begin{equation*}
\dim H^0(C_s, \mathcal L^{\otimes m_1}|_{C_s})\ne 0 
\end{equation*} holds 
for $m_1\geq 2(a+\dim X)$. Therefore, we have 
$C\not\subset \Bs_{\pi}|\mathcal L^{\otimes m_1}|$ for 
$m_1\geq 2(a+\dim X)$ by \eqref{a-eq8.1}. 
By this observation, 
we can take a desired effective Cartier divisor $D_1\in |\mathcal L^{\otimes m_1}|$ for 
every $m_1\geq 2(a+\dim X)$. 
\end{proof}

By the basepoint-free theorem 
(see Theorem \ref{a-thm6.1}), 
we have the commutative diagram: 
\begin{equation*}
\xymatrix{
X \ar[dr]_-\pi\ar[rr]^-p& & Z\ar[dl]^-q \\ 
&S&
}
\end{equation*}
such that $\mathcal L\simeq p^*\mathcal L_Z$ for some $q$-ample 
line bundle $\mathcal L_Z$ on $Z$ with $p_*\mathcal O_X\simeq \mathcal O_Z$. 
\begin{cla}\label{a-8.1-claim2}
If $\Bs_{\pi}|\mathcal L^{\otimes m_1}|$ contains no 
qlc strata of $[X, \omega]$, then 
\begin{equation*}
\dim \Bs_q|\mathcal L_Z^{\otimes m_2}|<
\dim \Bs_q|\mathcal L_Z^{\otimes m_1}|
\end{equation*}
holds for every positive integer 
$m_2\geq 2\left(a+(\dim X+1)m_1+\dim X\right)$ with 
$m_1|m_2$. 
\end{cla}
\begin{proof}[Proof of Claim \ref{a-8.1-claim2}]
We take general members $D_1, \ldots, D_{\dim X+1}\in 
|\mathcal L^{\otimes m_1}|=|p^*\mathcal L_Z^{\otimes m_1}|$. 
Let $\mathcal B$ be any irreducible component of 
$\Bs_q|\mathcal L_Z^{\otimes m_1}|$. 
Then, by Lemmas \ref{a-lem5.1} 
and \ref{a-lem5.2}, we can take $0<c<1$ such that 
$[X, \omega+cD]$, where $D:=\sum _{i=1}^{\dim X+1}D_i$, 
has a natural quasi-log structure with the following properties: 
\begin{itemize}
\item there exists a qlc center $V$ of $[X, \omega+cD]$ such that 
$p(V)=\mathcal B$, and 
\item $p(\Nqlc(X, \omega+cD))$ does not contain $\mathcal B$. 
\end{itemize} 
We put $V':=V\cup \Nqlc(X, \omega+cD)$. 
Then, by the proof of adjunction (see \eqref{a-eq4.3} in 
Theorem \ref{a-thm4.4}), 
we have the following short exact sequence: 
\begin{equation*}
0\to \mathcal I_{V'}\to \mathcal I_{\Nqlc(X, \omega+cD)}
\to \mathcal I_{\Nqlc(V', (\omega+cD)|_{V'})}\to 0. 
\end{equation*}
We note that $\mathcal L^{\otimes j}-(\omega+cD)$ is $\pi$-ample 
for every $j\geq a+(\dim X+1)m_1$. 
Therefore, 
\begin{equation*}
R^i\pi_*\left(\mathcal I_{V'}\otimes \mathcal L^{\otimes j}\right)
=R^i\pi_*\left(\mathcal I_{\Nqlc(X, \omega+cD)}
\otimes \mathcal L^{\otimes j}\right)
=R^i\pi_*\left(\mathcal I_{\Nqlc(V', (\omega+cD)|_{V'})}
\otimes \mathcal L^{\otimes j}\right)
=0
\end{equation*}
for every $i>0$ and $j\geq a+(\dim X+1)m_1$. 
In particular, we have the following short exact sequence: 
\begin{equation}\label{a-eq8.2}
\begin{split}
0\to \pi_*\left(\mathcal I_{V'}\otimes \mathcal L^{\otimes j}\right)
&\to \pi_*\left(\mathcal I_{\Nqlc(X, \omega+cD)}
\otimes \mathcal L^{\otimes j}\right) 
\\ &\to \pi_*\left(\mathcal I_{\Nqlc(V', (\omega+cD)|_{V'})}
\otimes \mathcal L^{\otimes j}\right)
\to 0
\end{split}
\end{equation}
for $j\geq a+(\dim X+1)m_1$. 
Let $V'\overset{p'}{\longrightarrow} V''\to p(V')$ be 
the Stein factorization of $V'\to p(V')$. 
By construction, 
we see that 
$p'_*\left(\mathcal I_{\Nqlc(V', (\omega+cD)|_{V'})}\right)$ is a non-zero 
ideal sheaf on $V''$. 
Therefore, 
\begin{equation*}
\pi_*\left(\mathcal I_{\Nqlc(V', (\omega+cD)|_{V'})}\otimes 
\mathcal L^{\otimes k}\right)\ne 0
\end{equation*} 
for every $k\gg 0$ since $\mathcal L\simeq p^*\mathcal L_Z$ and 
$\mathcal L_Z$ is $q$-ample. 
Let $s$ be an analytically sufficiently general point of $\pi(V)$. 
We put $V'_s:=V'|_{\pi^{-1}(s)}$. 
Thus, with at most $\dim V'_s$ exceptions, 
\begin{equation*}
\pi_*\left(\mathcal I_{\Nqlc(V', (\omega+cD)|_{V'})}\otimes 
\mathcal L^{\otimes (a+(\dim X+1)m_1+j)}\right)\ne 0
\end{equation*}
for $j\geq 0$. 
Hence, 
\begin{equation*}
\pi_*\left(\mathcal I_{\Nqlc(V', (\omega+cD)|_{V'})}\otimes 
\mathcal L^{\otimes m_2}\right)\ne 0
\end{equation*} 
for $m_2\geq 2\left(a+(\dim X+1)m_1+\dim X\right)$ by 
Lemma \ref{a-lem8.4}. 
Thus, by \eqref{a-eq8.2}, we obtain that 
$V\not \subset \Bs_{\pi}|\mathcal L^{\otimes m_2}|$ for 
every $m_2\geq 2\left(a+(\dim X+1)m_1+\dim X\right)$. 
This implies that $\mathcal B=p(V)\not \subset 
\Bs_q|\mathcal L_Z^{\otimes m_2}|$ for every 
$m_2\geq 2\left(a+(\dim X+1)m_1+\dim X\right)$. 
Therefore,  
\begin{equation*}
\dim \Bs_q|\mathcal L_Z^{\otimes m_2}|<
\dim \Bs_q|\mathcal L_Z^{\otimes m_1}|
\end{equation*} 
holds for 
$m_2\geq 2\left(a+(\dim X+1)m_1+\dim X\right)$ 
with $m_1|m_2$. 
This is what we wanted. 
We note that $\mathcal B$ is any irreducible component of $\Bs_q|\mathcal 
L_Z^{\otimes m_1}|$. 
\end{proof}
\end{step}
\begin{step}\label{a-8.1-step3} 
In this step, we will complete the proof. 

By Claim \ref{a-8.1-claim1}, we see that 
$\Bs_{\pi}|\mathcal L^{\otimes (2(a+\dim X))}|\subsetneq X$ and 
$\Bs_{\pi}|\mathcal L^{\otimes (2(a+\dim X))}|$ contains 
no qlc strata of $[X, \omega]$. 
Then we use Claim \ref{a-8.1-claim2} finitely many times. 
We finally obtain $m=m(\dim X, a)$, which 
only depends on $\dim X$ and $a$, 
such that $\mathcal L^{\otimes m(\dim X, a)}$ is $\pi$-generated. 
By Claims \ref{a-8.1-claim1}, \ref{a-8.1-claim2}, and Lemma \ref{a-lem8.5}, 
we can also take $m^\dag$, which only depends on 
$\dim X$ and $a$, such that 
$\mathcal L^{\otimes m^\dag}$ is $\pi$-generated 
with $\gcd(m(\dim X, a), m^\dag)=1$. 
Therefore, by Lemma \ref{a-lem8.6}, 
we can find a positive integer $m_0=m_0(\dim X, a)$ 
depending only on $\dim X$ and $a$ such that 
$\mathcal L^{\otimes l}$ is $\pi$-generated for every $l\geq m_0(\dim X, a)$. 
\end{step} 
We finish the proof. 
\end{proof}

\begin{proof}[Proof of Theorem \ref{a-thm8.2}]We take 
a positive integer $b$ with $b\geq a$ such that 
$b\omega$ is a line bundle on $X$. 
We put $\mathcal M:=a(b+1)\mathcal L-b\omega$. 
Then $\mathcal M$ is a $\pi$-nef line bundle 
such that $\mathcal M$ and $\mathcal M-\omega=
(b+1)(a\mathcal L-\omega)$ 
are nef and log big over $S$ with respect to $[X, \omega]$. 
After replacing $S$ with any relatively compact open subset of $S$, 
we have a contraction morphism $\varphi\colon X\to X'$ over 
$S$ associated to $\mathcal M$ and a contraction 
morphism $p\colon X\to Z$ over $S$ associated to $\mathcal L$ 
by the basepoint-free theorem of Reid--Fukuda type 
(see Theorem \ref{a-thm7.1}). 
Then we have the following commutative diagram: 
\begin{equation*}
\xymatrix{
& \ar[dl]_-\varphi X\ar[dr]^-p& \\ 
X' \ar[dr]_-{\pi'}\ar[rr]&& \ar[dl]^-q Z \\ 
&S&
}
\end{equation*} 
such that $\mathcal L\simeq \varphi^*\mathcal L'$ for some line bundle 
$\mathcal L'$ on $X'$ and $\omega\sim_{\mathbb Q} \varphi^*\omega'$ 
for some $\mathbb Q$-line bundle $\omega'$ on $X'$. 
Let $f\colon (Y, B_Y)\to X$ be a quasi-log resolution of $[X, \omega]$. 
Then 
\begin{equation*}
\left(X', \omega', \varphi\circ f\colon (Y, B_Y)\to X'\right)
\end{equation*} 
naturally becomes a quasi-log complex analytic space with 
$X'_{-\infty}=\emptyset$. 
By construction, $\mathcal L'$ is $\pi'$-nef 
over $S$ and $(a+1)\mathcal L'-\omega'$ is $\pi'$-ample 
over $S$ since $a+1\geq \frac{a(b+1)}{b}$. 
We note that 
\begin{equation*}
(a+1)\mathcal L'-\omega'\sim _{\mathbb Q} \frac{1}{b} \mathcal M'
+\left((a+1)-\frac{a(b+1)}{b}\right)\mathcal L', 
\end{equation*} 
where $\mathcal M\simeq \varphi^*\mathcal M'$. 
Thus, by Theorem \ref{a-thm8.1}, 
we obtain that $\mathcal L'^{\otimes m(\dim X, a+1)}$ is 
$\pi'$-generated and 
$\mathcal L'^{\otimes l}$ is $\pi'$-generated 
for every $l\geq m_0(\dim X, a+1)$. 
This implies the desired effective freeness for $\mathcal L\simeq 
\varphi^*\mathcal L'$. 
\end{proof}

Let us prove Theorem \ref{a-thm8.3}. 

\begin{proof}[Sketch of Proof of Theorem \ref{a-thm8.3}]
By assumption, $(a+1)\mathcal L-\omega$ is $\pi$-ample 
over $S$. 
We put 
\begin{equation*}
m'=m'(\dim X, a):=(\dim X+1)\times 
m(\dim X, a+1), 
\end{equation*} 
where $m(\dim X, a+1)$ is a positive integer obtained in 
Theorem \ref{a-thm8.1}. 
Then, we can check that 
$\mathcal L^{\otimes m'}$ is $\pi$-very ample 
(see, for example, \cite[Lemma 4.1]{fujino-effective} and 
\cite[Lemma 7.1]{fujino-effective-slc}). 
We put 
\begin{equation*}
m'_0=m'_0(\dim X, a):=m_0(\dim X, a+1)+m'(\dim X, a), 
\end{equation*} 
where $m_0(\dim X, a+1)$ is a positive integer obtained in 
Theorem \ref{a-thm8.1}. 
Then, it is easy to see that $\mathcal L^{\otimes l'}$ is $\pi$-very 
ample for every $l'\geq m'_0$. 
\end{proof}

Finally, we prove Theorem \ref{a-thm1.6}. 

\begin{proof}[Proof of Theorem \ref{a-thm1.6}]
We replace $S$ with any relatively compact open subset 
of $S$. 
Then $[X, K_X+\Delta]$ naturally becomes 
a quasi-log complex analytic space with 
$\Nqlc(X, K_X+\Delta)
=\emptyset$. By Theorems \ref{a-thm8.1} and 
\ref{a-thm8.2}, 
there exists a positive integer $m_0$ depending only on 
$\dim X$ and $a$ such that 
$\mathcal L^{\otimes m}$ is $\pi$-generated 
for every $m\geq m_0$. 
We note that $m_0$ is independent of 
$S$. 
Hence we see that $\mathcal L^{\otimes m}$ is 
$\pi$-generated for every $m\geq m_0$ without 
replacing $S$ with a relatively compact open subset 
of $S$. 
This is what we wanted. 
\end{proof}

\section{Cone theorem}\label{a-sec9}
In this section, we will briefly see that the 
cone and contraction theorem holds 
for quasi-log complex analytic spaces. 
We note that we have already established it 
for normal pairs in \cite[Theorem 1.1.6]{fujino-cone-contraction}. 
Let us start with the rationality theorem for 
quasi-log complex analytic spaces. 

\begin{thm}[Rationality theorem for quasi-log complex analytic 
spaces]\label{a-thm9.1}
Let $\pi\colon X\to S$ be a projective morphism 
of complex analytic spaces 
and let $W$ be a compact subset of $S$. 
Let $[X, \omega]$ be a quasi-log complex analytic space 
such that $\omega$ is a $\mathbb Q$-line bundle. 
Let $H$ be a $\pi$-ample 
line bundle on $X$. 
Assume that 
$\omega$ is not $\pi$-nef over $W$ and that 
$r$ is a positive real number such that 
\begin{itemize}
\item[(i)] $H+r\omega$ is $\pi$-nef 
over $W$ but is not $\pi$-ample over $W$, and 
\item[(ii)] $\left(H+r\omega\right)|_{\Nqlc(X, \omega)}$ 
is $\pi|_{\Nqlc(X, \omega)}$-ample over $W$. 
\end{itemize} 
Then $r$ is a rational number, and in reduced form, 
it has denominator at most $a(d+1)$, where 
$d:=\max_{w\in W}\dim \pi^{-1}(w)$ and 
$a$ is a positive integer such that 
$a\omega$ is a line bundle in a neighborhood of $\pi^{-1}(W)$. 
\end{thm}

There are no difficulties to adapt the proof of 
\cite[Theorem 4.3.1]{fujino-cone-contraction} for 
Theorem \ref{a-thm9.1}. 
So we omit the proof of Theorem \ref{a-thm9.1} here. 
By using the basepoint-free theorem (see Theorem 
\ref{a-thm6.1}) and 
the rationality theorem (see Theorem \ref{a-thm9.1}), 
we can establish the 
cone and contraction theorem for quasi-log complex analytic 
spaces. 

\begin{thm}[Cone and contraction theorem 
for quasi-log complex analytic spaces]\label{a-thm9.2}
Let $[X, \omega]$ be a quasi-log complex analytic space and 
let $\pi \colon X\to S$ be a projective 
morphism of complex 
analytic spaces. 
Let $W$ be a compact subset of $S$. 
We assume that the dimension of $N^1(X/S; W)$ is finite. 
Then we have 
\begin{equation*}
\NE(X/S; W)=\NE(X/S; W)_{\omega\geq 0} 
+\NE(X/S; W)_{\Nqlc (X, \omega)}+\sum_j R_j
\end{equation*} 
with the following properties.  
\begin{itemize}
\item[(1)] $\Nqlc (X, \omega)$ is the non-qlc locus 
of $[X, \omega]$ and 
$\NE(X/S; W)_{\Nqlc(X, \omega)}$ is the subcone 
of $\NE(X/S; W)$ which is 
the closure of the convex cone spanned by 
the projective integral curves $C$ on $\Nqlc(X, \omega)$ 
such that $\pi(C)$ is a point of $W$. 
\item[(2)] 
$R_j$ is an $\omega$-negative 
extremal ray of $\NE(X/S; W)$ which satisfies  
\begin{equation*}
R_j\cap \NE(X/S; W)_{\Nqlc (X, \omega)}=\{0\}
\end{equation*} for 
every $j$. 
\item[(3)] Let $\mathcal A$ be a $\pi$-ample 
$\mathbb R$-line bundle 
on $X$. 
Then there are only finitely many $R_j$'s included in 
$\NE(X/S; W)_{(\omega+\mathcal A)<0}$. In particular, 
the $R_j$'s are discrete in the half-space 
$\NE(X/S; W)_{\omega<0}$. 
\item[(4)] 
Let $F$ be any face of $\NE(X/S; W)$ such 
that 
\begin{equation*} 
F\cap \left(\NE(X/S; W)_{\omega\geq 0}+
\NE(X/S; W)_{\Nqlc (X, \omega)}\right)=\{0\}. 
\end{equation*}
Then, after shrinking $S$ around 
$W$ suitably, there exists a contraction morphism 
$\varphi_F\colon X\to Z$ 
over $S$ satisfying the following properties. 
\begin{itemize}
\item[(i)] 
Let $C$ be a projective integral curve on $X$ such that 
$\pi(C)$ is a point of $W$. 
Then $\varphi_F(C)$ is a point if and 
only if the numerical equivalence class 
$[C]$ of $C$ is in $F$. 
\item[(ii)] The natural map 
$\mathcal O_Z\to (\varphi_F)_*\mathcal O_X$ is an 
isomorphism.  
\item[(iii)] Let $\mathcal L$ be a line bundle on $X$ such 
that $\mathcal L\cdot C=0$ for 
every curve $C$ with $[C]\in F$. 
Then, after shrinking $S$ around 
$W$ suitably again, 
there exists a line bundle $\mathcal L_Z$ on $Z$ such 
that $\mathcal L\simeq \varphi^*_F\mathcal L_Z$ holds. 
\end{itemize}
\end{itemize}
\end{thm}

\begin{proof}[Sketch of Proof of Theorem \ref{a-thm9.2}]
The proof of the cone theorem for normal pairs 
in the complex analytic setting, which is described 
in \cite[Section 4.6]{fujino-cone-contraction}, 
works with only some minor modifications since 
we have already established the basepoint-free theorem 
(see Theorem \ref{a-thm6.1}) and 
the rationality theorem (see Theorem \ref{a-thm9.1}) 
for quasi-log complex analytic spaces. 
Note that we can use Lemma \ref{a-lem5.8} to 
reduce problems to the case where $\omega$ is 
a $\mathbb Q$-line bundle. 
\end{proof}

As an immediate application of Theorem 
\ref{a-thm9.2}, we have: 

\begin{thm}[Basepoint-freeness for $\mathbb R$-line bundles]\label{a-thm9.3}
Let $\pi\colon X\to S$ be a projective morphism 
between complex analytic spaces and let $W$ be a compact 
subset of $S$ such that the dimension of $N^1(X/S; W)$ is finite. 
Let $[X, \omega]$ be a quasi-log complex analytic space with 
$X_{-\infty}=\emptyset$ and let $\mathcal L$ be an $\mathbb R$-line 
bundle defined on some open neighborhood of $\pi^{-1}(W)$ 
such that $\mathcal L$ is $\pi$-nef over $W$. 
We assume that $a\mathcal L-\omega$ is $\pi$-ample over $W$ for 
some positive real number $a$. 
Then there exists an open neighborhood $U$ of 
$W$ such that $\mathcal L$ is $\pi$-semi-ample over $U$. 
\end{thm}

\begin{proof}[Sketch of Proof of Theorem \ref{a-thm9.3}]
Without loss of generality, we may assume that $a=1$ by replacing 
$\mathcal L$ with $a\mathcal L$. 
As in the proof of \cite[Theorem 5.3.1]{fujino-cone-contraction}, 
we can write $\mathcal L=\sum _{i=1}^m r_i \mathcal L_i$ such that 
\begin{itemize}
\item $\mathcal L_i$ is a $\mathbb Q$-line bundle for 
every $i$, 
\item $r_i$ is a positive real number for every $i$ with 
$\sum _{i=1}^m r_i=1$, and 
\item $\mathcal L_i-\omega$ is $\pi$-ample over $W$ for 
every $i$, 
\end{itemize} 
with the aid of the cone theorem (see Theorem \ref{a-thm9.2}). 
By the usual basepoint-free theorem 
(see Theorem \ref{a-thm6.1}), 
$\mathcal L_i$ is $\pi$-semi-ample over some open neighborhood 
of $W$ for every $i$. 
This implies that $\mathcal L$ is a finite positive 
$\mathbb R$-linear combination of $\pi$-semi-ample 
line bundles over some open neighborhood $U$ of $W$. 
This is what we wanted. 
\end{proof}

We have a supplementary result, which is 
a generalization of \cite[Theorem 1.1.8]{fujino-cone-contraction}. 

\begin{thm}\label{a-thm9.4} 
Let $[X, \omega]$ be a quasi-log complex analytic space with 
$X_{-\infty}=\emptyset$, that is, $[X, \omega]$ is a quasi-log 
canonical pair. 
Let $\pi\colon X\to S$ be a projective morphism 
of complex analytic spaces and let $W$ be a compact 
subset of $S$ such that the dimension of $N^1(X/S; W)$ is 
finite. 
Suppose that $\pi\colon X\to S$ is decomposed as 
\begin{equation*}
\xymatrix{
\pi\colon X\ar[r]^-f& S^\flat \ar[r]^-g& S
}
\end{equation*}
such that $S^\flat$ is projective over $S$. 
Let $\mathcal A_{S^\flat}$ be a $g$-ample line bundle on $S^\flat$. 
Let $R$ be an $\left(\omega+(\dim X+1)
f^*\mathcal A_{S^\flat}\right)$-negative 
extremal ray of $\NE(X/S; W)$. 
Then $R$ is 
an $\omega$-negative extremal ray of $\NE\left(X/{S^\flat}; 
g^{-1}(W)\right)$, that is, $R\cdot f^*\mathcal A_{S^\flat}=0$. 
\end{thm}

\begin{proof}[Sketch of Proof of Theorem \ref{a-thm9.4}]
We put $\mathcal L:=f^*\mathcal A_{S^\flat}$. 
We may assume that $\dim X\geq 1$. 
Since $\mathcal L$ is $\pi$-nef over $W$, 
$R$ is an $\omega$-negative extremal ray of $\NE(X/S; W)$. 
After shrinking $S$ around $W$ suitably, 
we obtain a contraction morphism 
$\varphi_R\colon X\to Z$ over $S$ associated to 
$R$ by the cone and contraction theorem 
(see Theorem \ref{a-thm9.2}). 
It is sufficient to prove that $\mathcal L\cdot R=0$ holds. 
We will get a contradiction by supposing that 
$\mathcal L\cdot R>0$ holds. 
If there exists a positive-dimensional 
irreducible component $X'$ of $X$ such that 
$\varphi_R(X')$ is a point. 
Then, by applying the argument in Step 1 in the 
proof of \cite[Theorem 1.1.8]{fujino-cone-contraction} to 
$[X', \omega|_{X'}]$, we get a contradiction. 
Hence, we may assume that 
$\dim \varphi_R(X')\geq 1$ holds for 
every positive-dimensional irreducible component $X'$ of $X$. 
By adjunction (see Theorem \ref{a-thm4.4}), 
$[X', \omega|_{X'}]$ is a quasi-log canonical 
pair. 
By considering $[X', \omega|_{X'}]$, we may assume that 
$X$ is irreducible. 
We take a point $P\in Z$ with $\varphi^{-1}_R(P)\geq 1$. 
By Lemma \ref{a-lem5.7}, 
after shrinking $Z$ around $P$ suitably, 
we can take an effective $\mathbb R$-Cartier divisor 
$G'$ on $Z$ such that 
$[X, \omega+\varphi^*_RG']$ naturally becomes 
a quasi-log complex analytic space, 
there exists a positive-dimensional 
qlc center $C$ of $[X, \omega+\varphi^*_RG']$ with 
$\varphi(C)=P$, 
$\dim \Nqlc(X, \omega+\varphi^*_RG')\leq 0$, 
and $\Nqlc(X, \omega
+\varphi^*_RG')=\emptyset$ outside $\varphi^{-1}_R(P)$. 
Then, by adjunction (see Theorem 
\ref{a-thm4.4}), we can construct a projective 
irreducible quasi-log scheme 
$[X'', \omega'':=\omega|_{X''}]$ such that 
$\varphi_R(X'')=P$, 
$-\omega''$ is ample, $\dim X''\geq 1$, and $\dim X''_{-\infty}\leq 0$. 
By construction, $\mathcal L|_{X''}$ is 
ample and $\omega''+r\mathcal L|_{X''}$ is 
numerically trivial for some $r>\dim X+1$. 
This is a contradiction by \cite[Lemma 4.7.1]{fujino-cone-contraction}. 
Anyway, we obtain that $\mathcal L\cdot R=0$. 
This is what we wanted.  
\end{proof}

As an obvious corollary of Theorem \ref{a-thm9.4}, 
we have: 

\begin{cor}\label{a-cor9.5}
Let $[X, \omega]$ be a quasi-log complex analytic space with 
$X_{-\infty}=\emptyset$, that is, $[X, \omega]$ is a quasi-log 
canonical pair. 
Let $\pi\colon X\to S$ be a projective morphism 
of complex analytic spaces 
and let $\mathcal A$ be any $\pi$-ample line bundle 
on $X$. Then $\omega+(\dim X+1)\mathcal A$ is 
always nef over $S$. 
\end{cor}

Corollary \ref{a-cor9.5} is a generalization 
of \cite[Corollary 1.1.9]{fujino-cone-contraction} 
(see also \cite[Corollaries 1.2 and 1.8]{fujino-relative-span}). 

\begin{proof}[Proof of Corollary \ref{a-cor9.5}]
Let $P\in S$ be any point. We put $W:=\{P\}$. 
Then we can check that the dimension of $N^1(X/S; W)$ is finite 
(see Remark \ref{a-rem2.7}). 
Assume that there 
exists an $\left(\omega+(\dim X+1)\mathcal A\right)$-negative 
extremal ray $R$ of $\NE(X/S; W)$. 
We apply Theorem \ref{a-thm9.4} by putting 
$S^\flat:=S$. Then we obtain $R\cdot \mathcal A=0$. 
This is a contradiction since $\mathcal A$ is 
ample. Therefore, there are no 
$\left(\omega+(\dim X+1)\mathcal A\right)$-negative 
extremal rays. This implies that $\omega+(\dim X+1)\mathcal A$ 
is $\pi$-nef over $W$. 
Since $P$ is any point of $S$, we obtain that 
$\omega+(\dim X+1)\mathcal A$ is nef over $S$. 
We finish the proof. 
\end{proof}

\subsection{On $\omega$-negative 
extremal rational curves}\label{a-subsec9.1}

The following results easily 
follow from \cite{fujino-cone}. 
They generalize \cite{kawamata}. 
We explicitly state them here for the reader's convenience. 
We think that \cite{fujino-cone} shows that 
the framework of quasi-log structures is useful. 
We note that the results in this subsection depend on 
Mori's bend and break method. 

\begin{thm}[{see \cite[Theorem 5.1.1]{fujino-cone-contraction}}]
\label{a-thm9.6}
Let $\varphi 
\colon X\to Z$ be a projective morphism of complex analytic spaces 
such that $[X, \omega]$ is a quasi-log complex analytic space. 
Assume that 
$-\omega$ is $\varphi$-ample. 
Let $P$ be an arbitrary point of $Z$. 
Let $E$ be any positive-dimensional irreducible component 
of $\varphi^{-1}(P)$ such that $E\not\subset \Nqlc(X, \omega)$. 
Then $E$ is covered by possibly singular rational curves $\ell$ 
with 
\begin{equation*}
0<-\omega\cdot \ell\leq 2 \dim E. 
\end{equation*} 
In particular, $E$ is uniruled. 
\end{thm}

\begin{proof}[Sketch of Proof of Theorem \ref{a-thm9.6}]
If $E$ is an irreducible component of 
$X$, then $[E, \omega|_E]$ is a projective 
quasi-log scheme by adjunction (see Theorem 
\ref{a-thm4.4}) and Lemma \ref{a-lem5.3} since 
$\varphi(E)=P$. 
Hence, the statement follows from 
\cite[Theorem 1.12]{fujino-cone}. 
From now on, we assume that 
$E$ is not an irreducible component of $X$. 
We take an irreducible component $X'$ of $X$ such that 
$E\subset X'$. 
By adjunction and Lemma \ref{a-lem5.3} again, 
$[X', \omega|_{X'}]$ is a quasi-log 
complex analytic space. 
By replacing $X$ with $X'$, we may assume that 
$X$ is irreducible. 
By Lemma \ref{a-lem5.6}, 
after shrinking $Z$ around $P$ suitably, 
we can take an effective $\mathbb R$-Cartier 
divisor $G$ on $Z$ such that $[X, \omega+\varphi^*G]$ 
naturally becomes a quasi-log complex analytic space and 
that $E$ is a qlc center of $[X, \omega+\varphi^*G]$. 
By adjunction (see Theorem 
\ref{a-thm4.4}) and Lemma \ref{a-lem5.3} 
(see also the proof of \cite[Theorem 5.1.1]{fujino-cone-contraction}), 
$[E, \omega|_E]$ is a projective 
quasi-log scheme such that $-\omega|_E$ is ample since 
$\varphi(E)=P$. 
Thus, by \cite[Theorem 1.12]{fujino-cone}, we have the desired 
properties. 
\end{proof}

Hence, we have: 

\begin{thm}[{Lengths of $\omega$-negative extremal rational curves, 
see \cite[Theorem 5.1.3]{fujino-cone-contraction}}]\label{a-thm9.7}
Let $\pi\colon X\to S$ be a projective morphism of 
complex analytic spaces such that 
$[X, \omega]$ is a quasi-log complex analytic space and 
let $W$ be a compact subset of $S$ 
such that the dimension of $N^1(X/S; W)$ is finite. 
If $R$ is an $\omega$-negative extremal ray of $\NE(X/S; W)$ 
which is relatively ample at $\Nqlc(X, \omega)$, 
that is, $R\cap \NE(X/S; W)_{\Nqlc(X, \omega)}=\{0\}$, 
then there exists a possibly singular 
rational curve $\ell$ spanning $R$ with 
\begin{equation*}
0<-\omega\cdot \ell\leq 2\dim X. 
\end{equation*}
\end{thm}

\begin{proof}[Sketch of Proof of Theorem \ref{a-thm9.7}]
By the cone and contraction theorem (see Theorem \ref{a-thm9.2}), 
we have a contraction morphism $\varphi_R\colon X\to Z$ over $S$ 
associated to $R$ after shrinking $S$ around $W$ suitably. 
By construction, $-\omega$ is $\varphi_R$-ample and 
$\varphi_R\colon \Nqlc(X, \omega)\to \varphi_R(\Nqlc(X, \omega))$ is finite. 
Thus, by Theorem \ref{a-thm9.6}, 
we can find a rational curve $\ell$ in a fiber of 
$\varphi_R$ with $0<-\omega\cdot \ell \leq 2\dim X$. 
This $\ell$ is a desired rational curve spanning $R$. 
\end{proof}

Theorem \ref{a-thm9.7} will play a crucial 
role in Subsection \ref{a-subsec10.2}. We close 
this subsection with a complex analytic 
generalization of \cite[Theorem 1.14]{fujino-cone}. 

\begin{thm}[{Rationally chain connectedness, 
see \cite[Theorem 1.14]{fujino-cone}}]\label{a-thm9.8} 
Let $\pi\colon X\to S$ be a projective morphism 
of complex analytic spaces with $\pi_*\mathcal O_X\simeq 
\mathcal O_S$ and let $[X, \omega]$ 
be a quasi-log complex analytic space. 
Assume that $-\omega$ is $\pi$-ample. 
Then $\pi^{-1}(P)$ is rationally chain connected modulo 
$\pi^{-1}(P)\cap X_{-\infty}$ for 
every point $P\in S$. 
In particular, if further $\pi^{-1}(P)\cap X_{-\infty}=\emptyset$ 
holds, that is, $[X, \omega]$ is quasi-log canonical in 
a neighborhood of $\pi^{-1}(P)$, then 
$\pi^{-1}(P)$ is rationally chain connected. 
\end{thm}

\begin{proof}
There are no difficulties to adapt 
the arguments in \cite[Section 13]{fujino-cone} 
to our complex analytic setting here. 
For the details, see \cite[Section 13]{fujino-cone}. 
\end{proof}

\section{On analytic semi-log canonical pairs}\label{a-sec10}

In this section, we will explain how to use the framework 
of quasi-log complex analytic spaces for the study of 
semi-log canonical pairs. 
In \cite{fujino-fundamental-slc}, the author proved that 
any quasi-projective semi-log canonical pair naturally has 
a quasi-log structure. The following theorem is a 
complex analytic generalization. 

\begin{thm}\label{a-thm10.1}
Let $(X, \Delta)$ be a semi-log canonical pair and 
let $\pi\colon X\to S$ be a projective morphism 
of complex analytic spaces. 
Then, after replacing $S$ with any relatively compact 
open subset of $S$, 
$[X, K_X+\Delta]$ naturally becomes a quasi-log complex 
analytic space such that 
$\Nqlc(X, K_X+\Delta)=\emptyset$ and 
that $C^\dag$ is a qlc center of $[X, K_X+\Delta]$ if and 
only if $C^\dag$ is a semi-log canonical center of $(X, \Delta)$. 

More precisely, after replacing $S$ with any relatively compact 
open subset of $S$, we can construct a projective surjective morphism 
$f\colon (Z, \Delta_Z)\to X$ from an analytic globally embedded 
simple normal crossing pair $(Z, \Delta_Z)$ such that 
the natural map 
\begin{equation*}
\mathcal O_X\to f_*\mathcal O_Z(\lceil 
-(\Delta^{<1}_Z)\rceil)
\end{equation*} 
is an isomorphism 
and that $C^\dag$ is 
the $f$-image of some stratum of $(Z, \Delta_Z)$ 
if and only if 
$C^\dag$ is a semi-log canonical 
center of $(X, \Delta)$ 
or an irreducible component of $X$. 
Moreover, if every irreducible component of $X$ has 
no self-intersection in codimension one, then we can 
make $f\colon Z\to X$ bimeromorphic. 
\end{thm}

\begin{proof}[Proof of Theorem \ref{a-thm10.1}]
We take an arbitrary relatively compact open subset $U$ 
of $S$. We will construct $f\colon (Z, \Delta_Z)\to X$ over $U$. 
In this proof, $X^{\ncp}$ denotes the largest open subset of $X$ consisting 
of smooth points, double normal crossing points, and 
pinch points. 
Similarly, $X^{\snc}$ denotes the largest open subset of $X$ consisting 
of smooth points and simple normal crossing points. 
Moreover, $X^{\snct}$ is the largest open subset of $X$ which 
has only smooth points and simple 
normal crossing points of multiplicity $\leq 2$. 
We note that $\Sing X$ denotes the singular locus of $X$. 
From Step \ref{a-10.1-step1} to Step \ref{a-10.1-step8}, 
we will explain how to construct a projective 
surjective morphism $f\colon (Z, \Delta_Z)\to X$ from 
an analytic globally embedded simple normal crossing pair 
$(Z, \Delta_Z)$ over $U$.  

\setcounter{step}{0}
\begin{step}\label{a-10.1-step1}
By \cite[Remark 1.6 and Theorem 1.18]{bierstone-milman3}, 
after replacing $S$ with any relatively compact open 
subset containing $\overline U$, 
we can take a morphism $f_1\colon X_1\to X$, which 
is a finite composite of admissible blow-ups, such that 
\begin{itemize}
\item[(i)] $X_1=X^{\ncp}_1$, 
\item[(ii)] $f_1$ is an isomorphism 
over $X^{\ncp}$, and 
\item[(iii)] $\Sing X_1$ maps bimeromorphically 
onto the closure of $\Sing X^{\ncp}$. 
\end{itemize}
We note that $X^{\ncp}=X^{\snct}$ and 
$X_1=X_1^{\ncp}=X_1^{\snct}$ hold 
in the above construction when every irreducible component 
of $X$ has no self-intersection in codimension one. 
\end{step}
\begin{step}\label{a-10.1-step2}
By replacing $S$ with any relatively compact open subset 
containing $\overline U$ again, 
$X_1$ is projective over $S$ by construction. 
Hence we can embed $X_1$ into $S\times \mathbb P^N$ 
over $S$ for some $\mathbb P^N$. 
We pick a finite set $\mathcal P\subset X_1$ such that 
each irreducible component of $\Sing X_1$ contains 
a point of $\mathcal P$. 
Moreover, we may assume that each irreducible component 
of $\Sing (\pi\circ f_1)^{-1}(U)$ contains 
a point of $\mathcal P$. 
After shrinking $S$ around 
$\overline U$, we take a sufficiently large positive integer $d$ such that 
$I_{X_1}\otimes \mathcal O(d)$ is globally generated, 
where $I_{X_1}$ is the defining ideal sheaf of $X_1$ on $S\times 
\mathbb P^N$ and 
$\mathcal O(d):=p^*\mathcal O_{\mathbb P^N}(d)$ with 
the second projection $p\colon S\times \mathbb P^N\to \mathbb P^N$. 
We take a complete intersection of 
$(\dim S+N-\dim X-1)$ general members of 
$|I_{X_1}\otimes \mathcal O(d)|$. 
Then we have $X_1\subset Y$ such that $Y$ is smooth at every point of 
$\mathcal P$. 
Note that here we used the fact that $X_1$ has only 
hypersurface singularities near $\mathcal P$. 
\end{step}
\begin{step}\label{a-10.1-step3}
After replacing $S$ with a relatively compact open subset 
containing $\overline U$ suitably, 
we take a resolution $g\colon Y_2\to Y$, which is a finite 
composite of admissible blow-ups and is an isomorphism 
over the largest Zariski open subset of $Y$ on which $Y$ is 
smooth (see \cite[Theorem 13.3]{bierstone-milman2}). Let $X_2$ be the 
strict transform of $X_1$ on $Y_2$. We note 
that $f_2:=g|_{X_2}\colon X_2\to X_1$ is an isomorphism 
over general points of any irreducible component of 
$\Sing X_1$ because $Y$ is smooth at every point of $\mathcal P$ 
by construction. 
\end{step}
\begin{step}\label{a-10.1-step4}
By applying \cite[Remark 1.6 and Theorem 1.18]{bierstone-milman3} 
to $X_2\subset Y_2$, after replacing $S$ with any relatively compact 
open subset containing $\overline U$, 
we have a finite composite of admissible blow-ups $g_3\colon Y_3\to Y_2$ 
from a smooth variety $Y_3$ such that 
$X_3=X^{\ncp}_3$ holds, where $X_3$ is the strict transform 
of $X_2$ on $Y_3$, and that 
$\Sing X_3$ maps bimeromorphically onto $\Sing X_1$ by $f_2\circ f_3$, 
where $f_3:=g_3|_{X_3}\colon X_3\to X_2$. 
We note that we can make $X_3=X_3^{\ncp} 
=X_3^{\snct}$ hold by 
\cite[Theorems 13.3 and 12.4]{bierstone-milman2} when 
$X_1=X_1^{\ncp} 
=X_1^{\snct}$ hold. 
\end{step}
\begin{step}\label{a-10.1-step5}
We put 
\begin{equation*}
K_{X_1}+\Delta_1=f^*_1(K_X+\Delta)
\end{equation*} 
and 
\begin{equation*}
K_{X_3}+\Delta_3=(f_1\circ f_2 \circ f_3)^*(K_X+\Delta). 
\end{equation*} 
Since $X_1$ and $X_3$ have only Gorenstein singularities, 
$\Delta_1$ and $\Delta_3$ are well-defined $\mathbb R$-Cartier 
$\mathbb R$-divisors on $X_1$ and $X_3$, respectively. 
By construction, the singular 
locus of $X_1$ (resp.~$X_3$) does not 
contain any irreducible components of 
$\Supp \Delta_1$ (resp.~$\Supp \Delta_3$). 
\end{step}

\begin{step}\label{a-10.1-step6}
Let $C$ be an irreducible component of $X_3\setminus X^{\snc}_3$. 
Then $C$ is smooth and $\dim C=\dim X_3-1$ since 
$X_3=X^{\ncp}_3$ holds. 
Let $\alpha\colon W\to Y_3$ be the blow-up along $C$ and 
let $V$ be $\alpha^{-1}(X_3)$ with 
the reduced structure. 
Then we can check that $\beta_*\mathcal O_V\simeq \mathcal O_{X_3}$, 
where $\beta:=\alpha|_V$. 
We put 
\begin{equation*} 
K_V+\Delta_V=\beta^*(K_{X_3}+\Delta_3). 
\end{equation*} 
We can easily check that $K_V=\beta^*K_{X_3}$ and 
$\Delta_V=\beta^*\Delta_3$. 
When $\alpha$ is the blow-up along a pinch points locus $C$, 
see \cite[Lemma 4.4]{fujino-fundamental-slc} for the 
local description of $\alpha\colon W\to Y_3$.  
When $\alpha$ is the blow-up along a double normal crossing 
points locus $C$, 
it is easy to understand $\alpha\colon W\to Y_3$. 
By repeating this process finitely many times 
and replacing $S$ with any relatively compact open 
subset containing $\overline U$, we obtain 
a projective bimeromorphic morphism 
$g_4\colon Y_4\to Y_3$ from a smooth variety $Y_4$ and a simple 
normal crossing divisor $X_4$ on $Y_4$ such that 
$f_{4*}\mathcal O_{X_4}\simeq \mathcal O_{X_3}$, where 
$f_4:=g_4|_{X_4}$. 
We put 
\begin{equation*}
K_{X_4}+\Delta_4:=f^*_4(K_{X_3}+\Delta_3). 
\end{equation*}
If $X_3=X_3^{\snct}$ holds, then 
$X_3\setminus X_3^{\snc}$ is empty. 
In this case, we have $Y_4=Y_3$ and $X_4=X_3$. 
\end{step}

\begin{step}\label{a-10.1-step7}
We consider the following closed subset $\Sigma:=\Supp 
(f_1\circ f_2\circ f_3\circ f_4)^*\Delta$ of $X_4$. 
By \cite[Theorems 13.3 and 12.4]{bierstone-milman2}, 
after replacing $S$ with any relatively compact open subset 
containing $\overline U$, 
we can construct a projective bimeromorphic 
morphism $g_5\colon Y_5\to Y_4$ with the following 
properties. 
\begin{itemize}
\item[(i)] Let $X_5$ be the strict transform of $X_4$ on $Y_5$. Then 
$f_5:=g_5|_{X_5}\colon X_5\to X_4$ is an isomorphism 
outside $\Sigma$ with $f_{5*}\mathcal O_{X_5}
\simeq \mathcal O_{X_4}$. 
\item[(ii)] $(X_5, \Sigma')$ is an analytic 
globally embedded simple normal crossing pair 
such that $\Sigma'$ is reduced and contains 
$\Supp f^{-1}_{5*}\Delta_4$ and $\Exc(f_5)$, where 
$\Exc(f_5)$ is the exceptional locus of $f_5$. 
\end{itemize}
\end{step}

\begin{step}\label{a-10.1-step8}
We replace $S$ with $U$ and 
put $M:=Y_5$, $Z:=X_5$, and $f:=f_1\circ f_2\circ f_3\circ f_4
\circ f_5\colon Z=X_5\to X$. 
We define $\Delta_Z$ by 
\begin{equation*}
K_Z+\Delta_Z:=f^*(K_X+\Delta). 
\end{equation*}
By construction, $f$ is a projective 
surjective morphism. 
Moreover, we see that $f$ is a projective bimeromorphic 
morphism when every irreducible component of $X$ has no 
self-intersection in codimension one. 
\end{step}

\begin{step}\label{a-10.1-step9} 
In this step, we will prove that $f_*\mathcal O_Z(\lceil -(\Delta_Z^{<1})\rceil)
\simeq \mathcal O_X$ holds. 

We first note that $X$ satisfies Serre's $S_2$ condition and 
$\codim _X(X\setminus X^{\ncp})\geq 2$ holds by assumption. 
Thus, we have $f_{1*}\mathcal O_{X_1}\simeq 
\mathcal O_X$. Since $\Delta$ is effective, 
$\lceil -(\Delta_1^{<1})\rceil$ is effective and $f_1$-exceptional, 
we see that 
$f_{1*}\mathcal O_{X_1}(\lceil -(\Delta_1^{<1})\rceil)\simeq 
\mathcal O_X$ holds. 
By construction, we can easily check that 
the following inclusions 
\begin{equation*}
\mathcal O_{X_1}\subset 
(f_2\circ f_3)_*\mathcal O_{X_3} (\lceil 
-(\Delta_3^{<1})\rceil)\subset 
\mathcal O_{X_1}(\lceil -(\Delta_1^{<1})\rceil)
\end{equation*} 
hold. 
Therefore, we obtain 
\begin{equation*}
(f_1\circ f_2\circ f_3)_* \mathcal O_{X_3}(\lceil 
-(\Delta_3^{<1})\rceil)\simeq \mathcal O_X. 
\end{equation*}
Let $\alpha\colon W\to Y_3$ be the blow-up 
in Step \ref{a-10.1-step6}. 
When $\alpha\colon W\to Y_3$ is the blow-up along a pinch points 
locus, see \cite[Lemma 4.4]{fujino-fundamental-slc} 
for the local description of $\alpha$. 
When $\alpha\colon W\to Y_3$ is the blow-up along a double 
normal crossing points locus, it is easy to describe $\alpha 
\colon W\to Y_3$. 
Then we have 
\begin{equation*}
0\leq \lceil -(\Delta_V^{<1})\rceil 
\leq \beta^*\left(\lceil -(\Delta_3^{<1})\rceil 
\right) 
\end{equation*} 
by $\Delta_V=\beta^*\Delta_3$. 
This implies 
\begin{equation*}
\mathcal O_{X_3}\subset \beta_*\mathcal O_V(\lceil -(\Delta_V^{<1})
\rceil )\subset \mathcal O_{X_3}(\lceil -(\Delta_3^{<1})\rceil) 
\end{equation*} 
by $\beta_*\mathcal O_V\simeq \mathcal O_{X_3}$. 
By using it finitely many times, we get 
\begin{equation*} 
\mathcal O_{X_3}\subset f_{4*}\mathcal O_{X_4} 
(\lceil -(\Delta_4^{<1})\rceil )\subset 
\mathcal O_{X_3}(\lceil -(\Delta_3^{<1})\rceil). 
\end{equation*} 
This implies 
\begin{equation*}
(f_1\circ f_2\circ f_3\circ f_4)_*\mathcal O_{X_4}
(\lceil -(\Delta_4^{<1})\rceil)\simeq 
\mathcal O_X. 
\end{equation*} 
It is easy to see that 
\begin{equation*}
\mathcal O_{X_4}\subset f_{5*}\mathcal O_{X_5} 
(\lceil -(\Delta_5^{<1})\rceil) \subset 
\mathcal O_{X_4}(\lceil -(\Delta_4^{<1})\rceil). 
\end{equation*} 
Thus, 
\begin{equation*}
(f_1\circ f_2\circ f_3\circ f_4\circ f_5)_*\mathcal 
O_{X_5}(\lceil -(\Delta_5^{<1})\rceil )\simeq 
\mathcal O_X. 
\end{equation*} 
This means that 
\begin{equation*}
f_*\mathcal O_Z(\lceil -(\Delta_Z^{<1})\rceil)\simeq 
\mathcal O_X. 
\end{equation*} 
This is what we wanted. 
\end{step}
\begin{step}\label{a-10.1-step10}
In this final step, we will see that $C^\dag$ is a semi-log canonical 
center if and only if $C^\dag$ is a qlc center of $[X, K_X+\Delta]$. 

When $f$ is bimeromorphic, the desired statement is almost obvious. 
Hence, we may assume that $f$ is not bimeromorphic. 
In this case, we can directly check the above statement 
with the aid of \cite[Lemma 4.4]{fujino-fundamental-slc}. 
\end{step}

We finish the proof. 
\end{proof}

By Theorem \ref{a-thm10.1}, 
we can apply the results for quasi-log complex 
analytic spaces to semi-log canonical 
pairs. 
We see that the basepoint-free theorem 
and its variants (see Theorems \ref{a-thm6.1}, 
\ref{a-thm7.1}, \ref{a-thm8.1}, \ref{a-thm8.2}, 
and \ref{a-thm8.3}) 
hold true for semi-log canonical pairs. 
All the results established in Section \ref{a-sec9} 
hold for semi-log canonical 
pairs. In particular, the cone and contraction theorem 
(see Theorem \ref{a-thm9.2}) holds for 
semi-log canonical pairs in the complex analytic setting. 

\subsection{Vanishing theorems and torsion-freeness for 
semi-log canonical pairs}\label{a-subsec10.1}
In this subsection, we will explicitly state the vanishing theorems 
and torsion-freeness for semi-log canonical pairs in the 
complex analytic setting for the reader's convenience. 

\begin{thm}[{\cite[Theorems 1.7 and 1.10]{fujino-fundamental-slc}}]
\label{a-thm10.2}
Let $(X, \Delta)$ be a semi-log canonical pair and 
let $\pi\colon X\to S$ be a projective morphism 
of complex analytic spaces. 
Let $D$ be a Cartier divisor on $X$, or a $\mathbb Q$-Cartier integral Weil 
divisor on $X$ such that no irreducible component of 
$\Supp D$ is contained in the singular locus of $X$. 
Assume that $D-(K_X+\Delta)$ is nef and log big over 
$S$ with respect to $(X, \Delta)$. 
This means that $D-(K_X+\Delta)$ is nef over 
$S$ and that $\left( D-(K_X+\Delta)\right)|_C$ is big over 
$\pi(C)$ for every slc stratum $C$ of $(X, \Delta)$. 
Then $R^i\pi_*\mathcal O_X(D)=0$ for every $i>0$. 
\end{thm}
\begin{proof}[Sketch of Proof of Theorem \ref{a-thm10.2}]
We take an arbitrary point $s\in S$. 
It is sufficient to prove $R^i\pi_*\mathcal O_X(D)=0$ for 
every $i>0$ on some open neighborhood of $s$. 
Hence, we can freely shrink $S$ around $s$ suitably. 
By \cite[5.23]{kollar-book}, we can take a double cover 
$p\colon \widetilde X\to X$ and reduce the problem to 
the case where every irreducible component of 
$X$ has no self-intersection in codimension one. 
Thus, we can construct a projective bimeromorphic 
morphism $f\colon (Z, \Delta_Z)\to X$ as in 
Theorem \ref{a-thm10.1}. 
By \cite[Theorems 13.3 and 12.4]{bierstone-milman2}, 
we may further assume that 
$(Z, \Sigma)$ is an analytic 
globally embedded simple normal crossing pair 
with $\Supp f^*D\cup \Supp \Delta_Z\subset \Sigma$. 
Thus, by Theorem \ref{a-thm3.5} (ii), 
we can apply the argument in the proof of 
\cite[Theorem 1.7 and Theorem 1.10]{fujino-fundamental-slc}. 
\end{proof}

\begin{thm}[{\cite[Theorem 1.11 and Remark 5.2]{fujino-fundamental-slc}}]
\label{a-thm10.3}
Let $(X, \Delta)$ be a semi-log canonical pair and 
let $\pi\colon X\to S$ be a projective morphism 
of complex analytic spaces. 
Let $\mathcal L$ be a line bundle on $X$. 
Let $X'$ be a union of some slc strata of $(X, \Delta)$ 
with the reduced structure and let $\mathcal I_{X'}$ be 
the defining ideal sheaf of $X'$ on $X$. 
Assume that $\mathcal L-(K_X+\Delta)$ is nef over 
$S$ and $\left(\mathcal L-(K_X+\Delta)\right)|_C$ is 
big over $\pi(C)$ for every slc stratum $C$ of $(X, \Delta)$ which 
is not contained in $X'$. 
Then $R^i\pi_*\left(\mathcal I_{X'}\otimes \mathcal L\right)=0$ holds 
for every $i>0$. 
In particular, $R^i\pi_*\mathcal L=0$ holds for 
every $i>0$ when $X'=\emptyset$. 
\end{thm}

\begin{proof}[Sketch of Proof of Theorem \ref{a-thm10.3}]
We take an arbitrary point $s\in S$. 
It is sufficient to prove that $R^i\pi_*\left(\mathcal I_{X'}\otimes 
\mathcal L\right)=0$ holds 
for every $i>0$ on 
some open neighborhood of $s$. 
By Theorem \ref{a-thm10.1}, 
after shrinking $S$ around $s$ suitably, 
we may assume that $[X, K_X+\Delta]$ is a quasi-log 
complex 
analytic space such that $C$ is a qlc stratum of 
$[X, K_X+\Delta]$ if and only if $C$ is an slc stratum of $(X, \Delta)$. 
Hence, we obtain $R^i\pi_*\left(\mathcal I_{X'}\otimes 
\mathcal L\right)=0$ for every $i>0$ 
by Theorem \ref{a-thm4.8}. 
\end{proof}

We close this subsection with the torsion-freeness, 
that is, the strict support condition,  
for semi-log canonical pairs. 

\begin{thm}[{\cite[Theorem 1.12]{fujino-fundamental-slc}}]\label{a-thm10.4}
Let $(X, \Delta)$ be a semi-log canonical pair and let $\pi\colon X\to S$ be 
a projective morphism of complex analytic spaces. 
Let $D$ be a Cartier divisor on $X$, or a $\mathbb Q$-Cartier 
integral Weil divisor on $X$ such that 
no irreducible component of $\Supp D$ is contained 
in the singular locus of $X$. 
Assume that $D-(K_X+\Delta)$ is $\pi$-semi-ample. 
Then every associated subvariety of $R^i\pi_*\mathcal O_X(D)$ 
is the $\pi$-image of some slc stratum of $(X, \Delta)$ for 
$i=0$ and $1$. 
\end{thm}

We make a very important remark on 
\cite[Theorem 1.12]{fujino-fundamental-slc}. 

\begin{rem}
[{Correction of \cite[Theorem 1.12]{fujino-fundamental-slc}}]\label{a-rem10.5}
In \cite[Theorem 1.12]{fujino-fundamental-slc}, 
we claim that every associated prime of $R^i\pi_*\mathcal O_X(D)$ 
is the generic point of the $\pi$-image of some slc stratum of 
$(X, \Delta)$ for {\em{every $i$}}. 
Unfortunately, however, the latter half of the proof of 
\cite[Theorem 1.12]{fujino-fundamental-slc} is insufficient. 
In the proof of \cite[Theorem 1.12]{fujino-fundamental-slc}, 
the $\pi|_A$-image of some slc stratum of $(A, \Delta|_A)$ 
is not necessarily the $\pi$-image of some slc stratum of $(X, \Delta)$. 
Hence, the correct statement of \cite[Theorem 1.12]{fujino-fundamental-slc} 
is that every associated prime of $R^i\pi_*\mathcal O_X(D)$ is 
the generic point of some slc stratum of $(X, \Delta)$ for $i=0$ and 
$1$. 
\end{rem}

\begin{proof}[Sketch of Proof of Theorem \ref{a-thm10.4}]
As in the proof of 
Theorem \ref{a-thm10.2}, 
the former half of the proof of \cite[Theorem 1.12]{fujino-fundamental-slc} 
works with some minor modifications. 
Hence we omit the details here. 
\end{proof}

\subsection{On Shokurov's polytopes for semi-log canonical pairs}
\label{a-subsec10.2}
In this final subsection, we will briefly explain Shokurov's 
polytopes for semi-log canonical pairs in the complex analytic setting. 
Here, we need the results in Subsection \ref{a-subsec9.1}. 
Therefore, the results in this subsection also depend on 
Mori's bend and break. 

Let $\pi\colon X\to S$ be a projective morphism 
between complex analytic spaces such that 
$X$ is equidimensional and let $W$ be 
a compact subset of $S$. 
Let $V$ be a finite-dimensional affine subspace of 
$\WDiv_{\mathbb R}(X)$, which is defined over the rationals. 
We put 
\begin{equation*} 
\mathcal L(V; \pi^{-1}(W)): 
=\{D\in V\, |\, (X, D) \ \text{is semi-log canonical at $\pi^{-1}(W)$}\}.  
\end{equation*}  
Then, as we saw in \ref{a-say2.4}, 
it is known that $\mathcal L(V; \pi^{-1}(W))$ is a 
rational polytope in $V$ defined over the rationals. 

\begin{defn}[Extremal curves]\label{a-def10.6} 
Let $\pi\colon X\to S$ 
be a projective morphism of complex analytic spaces and 
let $W$ be a compact 
subset of $S$ such that the dimension of $N^1(X/S; W)$ is finite. 
A curve $\Gamma$ on $X$ is called {\em{extremal over $W$}} 
if the following properties hold. 
\begin{itemize}
\item[(i)] $\Gamma$ generates an extremal ray $R$ of $\NE(X/S; W)$. 
\item[(ii)] There exists a $\pi$-ample 
line bundle $\mathcal H$ over some open neighborhood 
of $W$ such that 
\begin{equation*}
\mathcal H\cdot \Gamma=\min_{\ell} \{\mathcal H\cdot \ell\}, 
\end{equation*} 
where $\ell$ ranges over curves generating $R$. 
\end{itemize}
\end{defn}

By Theorem \ref{a-thm10.1} with Theorem \ref{a-thm9.7}, 
we have: 

\begin{lem}\label{a-lem10.7}
Let $\pi\colon X\to S$ be a projective morphism of complex 
analytic spaces and let $(X, \Delta)$ be a semi-log canonical pair. 
Let $W$ be a compact subset of $S$ such that 
the dimension of $N^1(X/S; W)$ is finite. 
Let $R$ be a $(K_X+\Delta)$-negative extremal ray 
of $\NE(X/S; W)$. 
If $\Gamma$ is an extremal curve over $W$ generating $R$, then 
\begin{equation*}
0<-(K_X+\Delta)\cdot \Gamma \leq 2\dim X
\end{equation*} 
holds. 
\end{lem}
\begin{proof}[Sketch of Proof of Lemma \ref{a-lem10.7}]
By Theorem \ref{a-thm10.1}, we may assume that 
$[X, K_X+\Delta]$ is a 
quasi-log canonical pair by shrinking $S$ around $W$ suitably. 
Then, by Theorem \ref{a-thm9.7}, we see that 
there exists a rational curve $\ell$ spanning $R$ such that 
\begin{equation*}
0<-(K_X+\Delta)\cdot \ell \leq 2\dim X. 
\end{equation*} 
Therefore, we obtain 
\begin{equation*}
0<-(K_X+\Delta)\cdot \Gamma 
=\left(-(K_X+\Delta)\cdot \ell\right)\cdot 
\frac{\mathcal H\cdot \Gamma}{\mathcal H\cdot \ell} 
\leq 2\dim X. 
\end{equation*} 
This is what we wanted. 
\end{proof}

We have already established the 
following theorems 
for log canonical pairs 
in \cite[Section 5.2]{fujino-cone-contraction}, 
which may be useful for the minimal model program 
with scaling. 

\begin{thm}\label{a-thm10.8}
Let $\pi\colon X\to S$ 
be a projective morphism of complex 
analytic spaces such that $X$ is equidimensional 
and let $W$ be a compact 
subset of $S$ 
such that the dimension of $N^1(X/S; W)$ is finite. 
Let $V$ be a finite-dimensional 
affine subspace of $\WDiv_{\mathbb R}(X)$, which 
is defined over the rationals. 
We fix an $\mathbb R$-divisor $\Delta\in \mathcal L(V; \pi^{-1}(W))$, 
that is, $\Delta\in V$ and $(X, \Delta)$ is semi-log 
canonical at $\pi^{-1}(W)$. 
Then we can find positive real numbers $\alpha$ and $\delta$, which 
depend on $(X, \Delta)$ and $V$, with the following properties. 
\begin{itemize}
\item[(1)] If $\Gamma$ is any extremal curve over $W$ and 
$(K_X+\Delta)\cdot \Gamma>0$, 
then $(K_X+\Delta)\cdot \Gamma >\alpha$. 
\item[(2)] If $D\in \mathcal L(V; \pi^{-1}(W))$, 
$|\!| D-\Delta|\!|<\delta$, and $(K_X+D)\cdot \Gamma \leq 0$ for 
an extremal curve $\Gamma$ over $W$, 
then $(K_X+\Delta)\cdot \Gamma\leq 0$. 
\item[(3)] Let $\{R_t\}_{t\in T}$ be any set of extremal rays of $\NE(X/S; W)$. 
Then 
\begin{equation*}
\mathcal N_T:=\{D\in \mathcal L(V; \pi^{-1}(W))\, |\, 
{\text{$(K_X+D)\cdot R_t\geq 0$ for every $t\in T$}}\} 
\end{equation*} 
is a rational polytope in $V$. 
In particular, 
\begin{equation*}
\mathcal N^\sharp _{\pi} (V; W):=\{\Delta\in 
\mathcal L (V; \pi^{-1}(W)) \, |\, 
{\text{$K_X+\Delta$ is nef over $W$}}\}
\end{equation*}
is a rational polytope. 
\end{itemize}
\end{thm}

\begin{proof}[Sketch of Proof of Theorem \ref{a-thm10.8}]
By Theorem \ref{a-thm10.1}, 
we may assume that $[X, K_X+\Delta]$ is a quasi-log 
canonical pair. 
Hence, we can use the cone and contraction 
theorem 
(see Theorem \ref{a-thm9.2}). 
Thus, we can apply the proof of \cite[Theorem 5.2.3]
{fujino-cone-contraction} with some minor modifications. 
We note that we need Lemma \ref{a-lem10.7} 
for the proof of (2) and (3). 
\end{proof}

We close this subsection with the following theorem, 
which is well known when $\pi\colon X\to S$ is 
algebraic and $X$ is a normal variety. 

\begin{thm}\label{a-thm10.9} 
Let $\pi\colon X\to S$ 
be a projective morphism of complex analytic spaces such that 
$X$ is equidimensional and let $W$ be a compact 
subset of $S$ such that the dimension of $N^1(X/S; W)$ is finite. 
Let $(X, \Delta)$ be a semi-log canonical pair and let $H$ be an effective 
$\mathbb R$-Cartier $\mathbb R$-divisor on $X$ such that 
$(X, \Delta+H)$ is semi-log canonical and that 
$K_X+\Delta+H$ is nef over 
$W$. 
Then, either $K_X+\Delta$ is nef over $W$ or there 
is a $(K_X+\Delta)$-negative extremal ray $R$ of $\NE(X/S; W)$ such that 
$(K_X+\Delta+\lambda H)\cdot R=0$, 
where 
\begin{equation*}
\lambda:=\inf \{t\geq 0
\, |\, {\text{$K_X+\Delta+tH$ is nef over $W$}}\}. 
\end{equation*} 
Of course, $K_X+\Delta+\lambda H$ is nef over $W$. 
\end{thm}

\begin{proof}[Sketch of Proof of Theorem \ref{a-thm10.9}]
As for Theorem \ref{a-thm10.8}, 
we can use the proof of 
\cite[Theorem 5.2.4]{fujino-cone-contraction} 
with only some minor modifications. So we omit the details 
here. 
\end{proof}

\end{document}